\documentclass[12pt,leqno]{amsart}  

\usepackage{calc}
\usepackage[left=1in,right=1in,top=1in,bottom=1in]{geometry}  
\usepackage{graphicx}
\usepackage{amsmath,amssymb,epsfig,mathtools}
\usepackage{amsthm}
\usepackage{enumerate}
\usepackage{caption}
\usepackage{tikz}
\makeatletter
\renewcommand*\env@matrix[1][*\c@MaxMatrixCols c]{%
  \hskip -\arraycolsep
  \let\@ifnextchar\new@ifnextchar
  \array{#1}}
\makeatother

\newcommand{\Z}{{\mathbb Z}}
\newcommand{\Q}{{\mathbb Q}}
\newcommand{\C}{{\mathbb C}}
\newcommand{\T}{{\mathcal T}}

\newcommand{\rank}{{\mathrm{rank}}}

\newcommand{\Sel}{{\mathrm{Sel}}}
\newcommand{\dimF}{{\mathrm{dim}_{\mathbb{F}_2}}}
\newcommand{\Ftwo}{{\mathbb{F}_2}}
\newcommand{\F}{{\mathbb{F}}}
\newcommand{\res}{{\mathrm{res}}}

\newcommand{\Zt}{{\mathbb{Z}/2\mathbb{Z}}}

\newcommand{\hatphi}{{ \hat \phi }}
\newcommand{\phihat}{{ \hat \phi }}

\newcommand{\ord}{{ \mathrm{ord} }}
\newcommand{\Max}{{ \mathrm{Max} }}

\newcommand{\legendre}[2]{\left( \frac{#1}{#2} \right)}
\newcommand{\legendrea}[2]{\left( \frac{#1}{#2} \right)_a}

\newcommand{\FF}{\mathbb{F}}
\newcommand{\RR}{\mathbb{R}}
\newcommand{\QQ}{\mathbb{Q}}

\newtheorem{thm}{\bf{Theorem}}
\newtheorem{theorem}{\bf{Theorem}}[section]
\newtheorem{cor}[thm]{\bf{Corollary}}
\newtheorem{corollary}[theorem]{\bf{Corollary}}
\newtheorem{prop}[theorem]{\bf{Proposition}}
\newtheorem{proposition}[theorem]{\bf{Proposition}}
\newtheorem{defn}[thm]{Definition}

\newtheorem{lem}[theorem]{\bf{Lemma}}
\newtheorem{lemma}[theorem]{\bf{Lemma}}

\newtheorem{claim}{\bf{Claim}}

\newcommand\coolover[2]{\mathrlap{\smash{%
\overbrace{\phantom{\begin{matrix} #2 %
\end{matrix}}}^{\mbox{$#1$}}}}#2}

\newcommand\coolleftbrace[2]{%
#1\left\{\vphantom{\begin{matrix}%
#2 \end{matrix}}\right.}

\theoremstyle{definition}
\newtheorem{definition}[theorem]{\bf{Definition}}
\newtheorem{remark}[theorem]{\bf{Remark}}

\input xy
\xyoption{all}

\usepackage[OT2,T1]{fontenc}

\DeclareSymbolFont{cyrletters}{OT2}{wncyr}{m}{n}
\DeclareMathSymbol{\Sha}{\mathalpha}{cyrletters}{"58}

\title{On the Joint Distribution Of $\Sel_\phi(E/\Q)$ and $\Sel_\hatphi(E^\prime/\Q)$ in Quadratic Twist Families}
\author{Daniel Kane} \author{Zev Klagsbrun}
\begin{document}
\bibliographystyle{alpha}

\begin{abstract}
If $E$ is an elliptic curve with a point of order two, then work of Klagsbrun and Lemke Oliver shows that the distribution of $\dimF \Sel_\phi(E^d/\Q) - \dimF \Sel_\hatphi(E^{\prime d}/\Q)$ within the quadratic twist family tends to the discrete normal distribution $\mathcal{N}(0,\frac{1}{2} \log \log X)$ as $X \rightarrow \infty$.

We consider the distribution of $\dimF \Sel_\phi(E^d/\Q)$ within such a quadratic twist family when $\dimF \Sel_\phi(E^d/\Q) - \dimF \Sel_\hatphi(E^{\prime d}/\Q)$ has a fixed value $u$. Specifically, we show that for every $r$, the limiting probability that $\dimF \Sel_\phi(E^d/\Q) = r$ is given by an explicit constant $\alpha_{r,u}$. The constants $\alpha_{r,u}$ are closely related to the $u$-probabilities introduced in Cohen and Lenstra's work on the distribution of class groups, and thus provide a connection between the distribution of Selmer groups of elliptic curves and random abelian groups.

Our analysis of this problem has two steps. The first step uses algebraic and combinatorial methods to directly relate the ranks of the Selmer groups in question to the dimensions of the kernels of random $\mathbb{F}_2$-matrices. This proves that the density of twists with a given $\phi$-Selmer rank $r$ is given by $\alpha_{r,u}$ for an unusual notion of density. The second step of the analysis utilizes techniques from analytic number theory to show that this result implies the correct asymptotics in terms of the natural notion of density.

\end{abstract}
\maketitle

\vspace{-0.2in}

\section{Introduction}
Recently, there has a lot of interest in the arithmetic statistics related to the quadratic twist family of a given elliptic curve $E/\Q$. Much progress has been made towards understanding how 2-Selmer ranks are distributed in these families when either $E(\Q)[2] \simeq \Zt \times \Zt$ or $E[2]$ has an $\mathcal{S}_3$ Galois action. In both of these cases, there are explicit constants $\alpha_r$ summing to one such that the proportion of twists with 2-Selmer rank $r$ is given by $\alpha_r$ \cite{Kane}, \cite{KMR2}.





Strikingly, this is not true when $E$ has a single rational point of order two. In this case $E$ has a degree two isogeny $\phi:E \rightarrow E^\prime$ and an associated Selmer group $\Sel_\phi(E/\Q)$. Work of Xiong shows that if $E$ does not have a cyclic 4-isogeny defined over $\Q(E[2])$, then the distribution of the ranks of $\Sel_\phi(E^d/\Q)$ as $d$ varies among the squarefree integers less than $X$ tends to the distribution $\Max \left( 0 , \mathcal{N}(0, \frac{1}{2} \log\log X) \right )$ as $X \rightarrow \infty$, where $\mathcal{N}(\mu, \sigma^2)$ is the discrete normal distribution with mean $\mu$ and variance $\sigma^2$ \cite{X}. In this case, $\Sel_\phi(E/\Q)$ maps $2 - to - 1$ into $\Sel_2(E/\Q)$, showing that for any fixed $r$, at least half of the quadratic twists of $E$ have 2-Selmer rank greater than $r$.

This same result can be deduced by studying how $\dimF \Sel_\phi(E/\Q) - \dimF \Sel_\hatphi(E^{\prime }/\Q)$ varies under quadratic twist, where $\Sel_\hatphi(E^\prime/\Q)$ is the Selmer group associated to the dual isogeny $\hatphi$ of $\phi$. In \cite{KLO}, Lemke Oliver and the second author shows that as $d$ varies among the squarefree integers less than $X$, the distribution of $\dimF \Sel_\phi(E^d/\Q) - \dimF \Sel_\hatphi(E^{\prime d}/\Q)$ tends to $\mathcal{N}(0,\frac{1}{2} \log \log X)$ as $X \rightarrow \infty$.





This article studies the joint distribution of $\Sel_\phi(E^d/K)$ and $\Sel_\hatphi(E^{\prime d}/K)$ conditional on a fixed value of $\dimF \Sel_\phi(E^d/K) - \dimF \Sel_\hatphi(E^{\prime d}/K)$. In particular, we prove the following:

\begin{thm}\label{MainThm}
Suppose $E/\Q$ is an elliptic curve with $E(\Q)[2] \simeq \Zt$ that does not have a cyclic 4-isogeny defined over $\Q(E[2])$ and $u \in \Z$. Define $$S(X, u) =  \{ d \text{ squarefree }, |d| \le X,  \dimF \Sel_\phi(E^d/\Q) - \dimF \Sel_\hatphi(E^{\prime d}/\Q) = u \}.$$ Then for any $r \ge \Max (1, u+1)$, $$\lim_{X \rightarrow \infty} \frac{|\{ d \in S(X,u) : (\dimF \Sel_\phi(E^d/\Q), \dimF  \Sel_\hatphi(E^{\prime d}/\Q) ) = (r, r-u) \}|}{|S(X,u)|} = \alpha_{r,u},$$ where $$\alpha_{r,u} = \frac{2^{-(r-1)(r-u-1)}\prod_{s = 1}^\infty(1 - 2^{-s})}{\prod_{s = 1}^{r-1}(1-2^{-s}) \prod_{s = 1}^{r-u-1}(1-2^{-s}) }.$$
\end{thm}

Theorem \ref{MainThm} is similar to the results of Thorne and the first author regarding the distribution of $\phi$-Selmer groups in the $j = 1728$ family of elliptic curves \cite{KT}.

\subsection{Connections With the Cohen-Lenstra Heuristics}

In 1984, Cohen and Lenstra conjectured that if $K$ is an imaginary quadratic field, then the probability that $Cl(K)[p^\infty]$ is isomorphic to a fixed finite abelian $p$-group $G$ should be proportional to $\frac{1}{|Aut(G)|}$. This conjecture infers a distribution on the $p$-rank of $Cl(K)$ and Washington observed that this distribution is identical to one appearing in random matrix theory \cite{WashCL}. Assuming the Cohen-Lenstra heuristic, the probability that $Cl(K)[p^\infty]$ has $p$-rank $r$ is the same as the probability that a random $n \times n$ matrix over $\FF_2$ has nullity $r$ as $n \rightarrow \infty$ \cite{FulmanGoldstein}.

In their original paper, Cohen and Lenstra also defined a notion of the $u$-probability of a group $G$. Let $H$ be a random $p$-group $H$ chosen with probability proportional to $\frac{1}{|Aut(H)|}$ and $h_1,h_2,\ldots,h_u$ be elements of $H$ chosen uniformly at random. The \textit{$u$-probability of $G$} is the probability that $H/\langle h_1,h_2, \ldots,h_u\rangle \simeq G$. There is a similar notion for $p$-ranks and Cohen and Lenstra obtain the following result.
\begin{theorem}[Theorem 6.3 in \cite{CL}]\label{thm:CL6.3}
Define the \textit{$u$-probability that $H$ has rank $r$} as the probability that $\rank_p(H/\langle h_1,h_2,\ldots,h_u\rangle) = r$. The $u$-probability that a $p$-group $H$ has rank $r$ is given by \begin{equation}\label{eq:alphaprime}\alpha^\prime_{p,u,r} = \frac{p^{-r(r+u)} \prod_{s = 1}^\infty (1 - p^{-s})}{\prod_{s = 1}^r (1 - p^{-s})\prod_{s = 1}^{r+u} (1 - p^{-s})}\end{equation}
\end{theorem}
While the notion of $u$-probability is only sensible for $u \ge 0$, we may nonetheless extend the definition to include $u \le 0$ by defining it to be $\alpha^\prime_{p,u,r}$ as in (\ref{eq:alphaprime}) if $r \ge u$ and zero otherwise. As can be seen, the contants $\alpha_{r,u}$ in Theorem \ref{MainThm} are given by $\alpha_{r,u} = \alpha^\prime_{2,r+1,-u}$. That is, if $u =  \dimF \Sel_\phi(E^d/\Q) - \dimF \Sel_\hatphi(E^{\prime d}/\Q)$, then for any $r \ge \Max(1,u+1)$,  the probability that $ \dimF \Sel_\phi(E^d/\Q) = r$ is equal to the $-u$ probability that a random $2$-group $H$ has rank $r+1$. Other than the related results in \cite{KT}, this is the only instance in which these $\alpha^\prime_{p,u,r}$ have been provably shown to arise in the context of arithmetic statistics.

\subsection{Methods and Organization}

The constants $\alpha^\prime_{p,u,r}$ in Theorem \ref{thm:CL6.3} appear in the following well-known theorem from random matrix theory.
\begin{theorem}\label{thm:randmatthm}
Let $M$ be a randomly chosen $n \times n + u$ matrix over $\FF_p$. Then the probabilty that the left nullspace of $M$ has dimension $r$ tends to $\alpha^\prime_{p,r,u}$ exponentially quickly as $n \rightarrow \infty$.
\end{theorem}
\begin{proof}
This limiting behavior was known at least as far back as \cite{KLS}. The fact that this convergence is exponential in $n$ follows from Theorem 1.1 in \cite{FulmanGoldstein}, for example.
\end{proof}

We obtain Theorem \ref{MainThm} by relating the problem to a question about random matrices over $\FF_2$ and then applying Theorem \ref{thm:randmatthm}. Our proof proceeds as follows:

As described in Sections \ref{sec:descent}-\ref{sec:selcoker}, we equate the dimension of a co-dimension one subgroup of $\Sel_\phi(E/\QQ)$ with the dimension of the left-nullspace of an $n \times n-u$ matrix $\mathcal{\widehat{M}}$ with entries in $\FF_2$. If the entries of $\mathcal{\widehat{M}}$ were independent and random, then we would be done. Unsurprisingly however, there are dependencies between the entries in $\mathcal{\widehat{M}}$. Nonetheless, in Section \ref{sec:probapproach}, we show that under some mild assumptions regarding the values of certain characters involving $d$, $\mathcal{\widehat{M}}$ is equivalent to a block diagonal matrix $\begin{bmatrix}I & 0 \\ 0 & A \end{bmatrix}$ with high probability, where $I$ is an $n_0 \times n_0$ identity matrix and $A$ is an $n - n_0 \times n - n_0 -u$ matrix with independent random entries. Section \ref{sec:analytictech} then uses techniques from analytic number theory to show that the assumptions we made regarding the characters involving $d$ are satisfied with sufficiently high probability. As a result, we obtain Theorem \ref{MainThm}.

\subsection{Acknowledgements}

We would like to thank Benedek Valko for explaing to us how a result similar to Theorem \ref{thm:SDLimitThm} may be obtained via a generalization of the Markov chain approach developed in \cite{KMR2}. We would also like to thank Jordan Ellenberg for pointing out the relationship between the constants $\alpha_{r,u}$ and the notion of $u$-probabilities in the work of Cohen and Lenstra.


\section{$\phi$-Descent}\label{sec:descent}

We begin by defining the Selmer groups $\Sel_\phi(E/\Q)$ and $\Sel_\hatphi(E^{\prime}/\Q)$ and then giving an explicit description of the Selmer groups $\Sel_\phi(E^d/\Q)$ and $\Sel_\hatphi(E^{\prime d}/\Q)$ associated to the quadratic twist of an elliptic curve by a squarefree integer $d$.

Let $E$ be an elliptic curve with a single point of order two defined by $$y^2 = x^3 + Ax^2 + Bx.$$ and set $C = E(\Q)[2] = \langle (0,0) \rangle$. There is an isogenous curve $E^\prime$ given by a model $$y^2 = x^3 - 2Ax^2 + (A^2- 4B)x$$ and an isogeny $\phi:E\rightarrow E^\prime$ with kernel $C$. There is a Kummer map $$\kappa: E^\prime(\Q)/\phi(E(\Q)) \xrightarrow{\sim} \Q^\times/(\Q^\times)^2$$ given by $$\kappa((x,y)) = \left \{ \begin{array}{cl}\Delta  & \text{if  } (x,y) = (0,0) \\  x  & \text{if } (x,y) \ne (0,0)  \end{array} \right . $$ where $\Delta$ is the discriminant of $E$.

We have similarly defined local Kummer maps $$\kappa_v: E^\prime(\Q_v)/\phi(E(\Q_v)) \xrightarrow{\sim} \Q_v^\times/(\Q_v^\times)^2$$ for every completion $\Q_v$ of $\Q$ which give a commutative diagram for every place $v$ of $\Q$, where the restriction map $\res_v$ is the natural map $\Q^\times/(\Q^\times)^2 \rightarrow \Q_v^\times/(\Q_v^\times)^2$.

\begin{center}\leavevmode
\begin{xy} \xymatrix{
E^\prime(\Q)/\phi(E(\Q))  \ar[d] \ar[r]^{\hspace{0.2in}\kappa} & \Q^\times/(\Q^\times)^2 \ar[d]^{\res_v} \\
E^\prime(\Q_v)/\phi(E(\Q_v))   \ar[r]^{\hspace{0.25in}\kappa_v} & \Q_v^\times/(\Q_v^\times)^2 }
\end{xy}\end{center}

The $\phi$-Selmer group $\Sel_\phi(E/\Q)$ is defined as  \begin{equation*}\Sel_\phi(E/\Q) = \left \{ c \in  \Q^\times/(\Q^\times)^2 : \res_v(c) \in \kappa_v(E^\prime(\Q_v)/\phi(E(\Q_v)) ) \text{ for all places } v \text{ of } \Q \right \} \end{equation*}

If $p$ is a prime away from $2$ where $E$ has good reduction, the image of $\kappa_p$ is equal to the unramified subgroup of $\Q_p^\times/(\Q_p^\times)^2$ generated by the image of $\Z_p^\times$. This allows us to decribe the $\Sel_\phi(E/\Q)$ as the intersection of two finite dimensional $\Ftwo$ vector spaces.

Let $T$ be the set of places of $\Q$ dividing $2\Delta\infty$ and define $$V = \bigoplus_{v \in T} \Q_v^\times/(\Q_v^\times)^2.$$ Define a subspace $U \subset V$ as the image of the $T$-units $\Z_T^\times$ in $V$. Next, for each place $v \in T$, define $W_v$ as $$W_v = \kappa_v(E^\prime(\Q_v)/\phi(E(\Q_v)) )$$ and set $$W  = \bigoplus_{v \in T} W_v \subset V$$ It then follows that $\Sel_\phi(E/\Q) \simeq U \cap W$.

Exchanging the roles of $E$ and $\phi$ for those of $E^\prime$ and the dual isogeny $\hatphi:E^\prime \rightarrow E$ yields a $\hatphi$-Selmer group $\Sel_\hatphi(E^\prime/\Q)$ via the same construction.




\subsection{Tamagawa Ratios}

Standard descent technology tells us that the images of the local connecting maps $\kappa_p$ and $\kappa_p^\prime$ in $\Q_p^\times/(\Q_p^\times)^2$ are dual to each other via the Hilbert symbol pairing. 
It follows that the Selmer groups $\Sel_\phi(E/\Q)$ and $\Sel_\hatphi(E^\prime/\Q)$ are orthogonal under the sum of the Hilbert symbol pairings over the places in $T$. This duality gives us a way to compare the sizes of $\Sel_\phi(E/\Q)$ and $\Sel_\hatphi(E^\prime/\Q)$.

\begin{definition} The ratio $$\mathcal{T}(E/E^\prime) = \frac{\big | \Sel_\phi(E/\Q) \big |}{\big |\Sel_{\hatphi}(E^\prime/\Q)\big |}$$ is called the \textbf{Tamagawa ratio} of $E$. \end{definition}

The Tamagawa ratio can be computed using a local product formula.

\begin{theorem}[Cassels]\label{prodform2}
The Tamagawa ratio $\mathcal{T}(E/E^\prime)$ is given by $$\mathcal{T}(E/E^\prime) = \prod_{v\text{ of } \Q}\frac{|W_v|}{2}.$$
\end{theorem}
\begin{proof}
This is a combination of Theorem 1.1 and equations (1.22) and (3.4) in \cite{Cassels8}. Alternatively, this follows from combining Theorem 2 in \cite{W} with the fact that the images of $\kappa_v$ and $\kappa^\prime_v$ are dual to each other.
\end{proof}

\section{Twisting}\label{sec:twisting}

Let $d$ be a squarefree integer and set $T_d = T \cup \{p \mid d\}$.  The set $T_d$ contains all of the places of $\Q$ above $2\infty$ and the places at which $E^d$ has bad reduction.
We define $$V^d = \bigoplus_{v \in T_d} \Q_v^\times/(\Q_v^\times)^2$$ and define $U^d \subset V^d$ as the image of $\Z_{T_d}^\times$ in $V^d$.  For each place in $v \in T_d$, we define $$W_v^d =  \kappa_v(E^{\prime d}(\Q_v)/\phi(E^d(\Q_v)) )$$%
and set $$W^d = \bigoplus_{v \in T_d} W_v^d \subset V^d.$$%
We then get that $\Sel_\phi(E^d/\Q) \simeq U^d \cap W^d$.

\subsection{Local Conditions at Twisted Primes}

If $p \mid d$ and $(p,2\Delta) = 1$, we can explicitly describe the subspace $W_p^d$ of $ \Q_p^\times/(\Q_p^\times)^2$.

\begin{lemma}\label{lem:loccond}
If $p \mid d$ and $(p,2\Delta) = 1$, then $$W_p^d = \left \{ \begin{array}{cl} \langle \Delta \rangle  & \text{if  } \left ( \frac{\Delta^\prime}{p} \right ) = -1 \\  \langle \Delta, d(A+2\sqrt{B}) \rangle & \text{if } \left ( \frac{\Delta^\prime}{p} \right ) = 1 \end{array} \right . ,$$ where $\Delta^\prime$ is the discriminant of $E^\prime$.
\end{lemma}
\begin{proof}
By Lemma 6.7 in \cite{K}, the image of $\kappa_p$ is given by $\kappa_p(E^{\prime d}(\Q_p)[2])$. The result then follows from symbolically computing the coordinates of $E^{\prime d}[2]$.
\end{proof}

\begin{remark}
Up to squares, we have $\Delta = (A^2 - 4B)(\Q^\times)^2$ and $\Delta^\prime = B(\Q^\times)^2$.
\end{remark}

We may be even more explicit about $W_p^d$
if we characterize the primes $p \mid d$ with $(p,2\Delta)=1$ by the values of the Legendre symbols $\legendre{\Delta}{p}$ and $\legendre{\Delta^\prime}{p}$.

\begin{definition}
Suppose that $(p,2\Delta) = 1$. \\
We say that $p$ is \emph{type 1} if $\legendre{\Delta}{p}= 1$ and $\legendre{\Delta^\prime}{p}=1$.\\
We say that $p$ is \emph{type 2} if $\legendre{\Delta}{p}=1$ and $\legendre{\Delta^\prime}{p}=-1$.\\
We say that $p$ is \emph{type 3} if $\legendre{\Delta}{p} = -1$ and $\legendre{\Delta^\prime}{p}=1$.\\
We say that $p$ is \emph{type 4} if $\legendre{\Delta}{p}=-1$ and $\legendre{\Delta^\prime}{p}=-1$.
\end{definition}

We now note that $W^d_p$ is dependent on the type of $p$.

\begin{corollary}\label{cor:locconds} If $p \mid d$ and $(p,2\Delta) = 1$, then
$$
W^d_p = \begin{cases}
\langle d(A+2\sqrt{B})\rangle & \textrm{if }p\textrm{ is of type 1}\\
1 & \textrm{if }p\textrm{ is of type 2}\\
\Q_p^\times/(\Q_p^\times)^2 & \textrm{if }p\textrm{ is of type 3}\\
\Z_p^\times/(\Z_p^\times)^2 & \textrm{if }p\textrm{ is of type 4}\\
\end{cases}
$$
\end{corollary}
\begin{proof}
This follows from Lemma \ref{lem:loccond} and the isomorphism $\Q_p^\times/(\Q_p^\times)^2 \xrightarrow{\sim} \Z_p^\times/(\Z_p^\times)^2 \times \langle p \rangle /\langle p^2 \rangle$.
\end{proof}

\begin{corollary}\label{cor:tambytype}
The valuation of $2$ in the Tamagawa ratio $\T(E^d/E^{\prime d})$ is given by $$\ord_2  \T(E^d/E^{\prime d}) = \sum_{v \mid 2\Delta\infty} (\dimF W_v^d -1) + |T_d^3| - |T_d^2|,$$ where $T_d^i = \{ p \mid d: (p,2\Delta) = 1 \text{ and $p$ is over type $i$ } \}$.
\end{corollary}
\begin{proof}
This follows from Theorem \ref{prodform2} combined with Corollary \ref{cor:tambytype}.
\end{proof}

\begin{proposition}\label{prop:tamcalc}
The subspace $W_v^d$ is determined by the image of $d \in \Q_v^\times/(\Q_v^\times)^2$. Therefore, $$\ord_2 \T(E^d/E^{\prime d}) = c_d + |T_d^3| - |T_d^2|$$ where $c_d$ is determined entirely by the image of $d$ in $\prod_{v \mid 2\Delta\infty} \Q_v^\times/(\Q_v^\times)^2$.
\end{proposition}
\begin{proof}
The image of $\kappa_v$ in $\Q_v^\times/(\Q_v^\times)^2$ is dependent only on the isomorphism class of $E^d$ over $\Q_v$. As a result, this image is the same for all $d$ with a given image in  $\Q_v^\times/(\Q_v^\times)^2$. The result then follows from Corollary \ref{cor:tambytype}.
\end{proof}

\section{Selmer Groups as Cokernels}\label{sec:selcoker}

Let $E^d$ be the quadratic twist of a fixed elliptic curve $E$ by $d$. The point $(0,0) \in E^{\prime d}(\Q)$ maps to $\Delta \in \Sel_\phi(E/\Q) \subset \Q^\times/(\Q^\times)^2$. Since $E$ is assumed to have a single point of order two, we find that $\Delta$ is not a square in $\Q^\times$. Let $\Sel_\phi(E^d/\QQ)_0$ be any co-dimension one subspace of $\Sel_\phi(E^d/\QQ)$ not containing $\Delta$. We then have $\Sel_\phi(E^d/\QQ) = \Sel_\phi(E^d/\QQ)_0 \oplus \Delta\Sel_\phi(E^d/\QQ)_0$. The goal of this section is produce a matrix $\mathcal{M}$ such that one such $\Sel_\phi(E^d/\Q)_0$ may be identified with the left-nullspace of $\mathcal{M}$.

Having produced $\mathcal{M}$, we will describe a process called \textit{surgery} to deform $\mathcal{M}$ into a matrix $\widehat{\mathcal{M}}$ such that the left nullspaces of $\mathcal{M}$ and $\widehat{\mathcal{M}}$ have the same dimension. The entries of $\widehat{\mathcal{M}}$ will be easier to describe that those of $\mathcal{M}$ and its dimension will be easier to model.

For this section, we will assume that $d$ is squarefree and $(d,2\Delta) = 1$. We will also adopt the notation $\legendrea{\cdot}{p}$ to refer to the additive Legendre character taking values in $\FF_2$.

\subsection{An Easy Presentation}\label{sec:easypres}

As described in Section \ref{sec:twisting}, $\Sel_\phi(E/\QQ)$ is given by the intersection $\Sel_\phi(E^d/\QQ) = U^d \cap W^d \subset V^d$. If $U_0^d$ is any co-dimension one subspace of $U^d$ that does not contain $\Delta$, then we may take $\Sel_\phi(E^d/\Q)_0$ to be $U_0^d \cap W^d$.

A natural presentation for $\Sel_\phi(E^d/\QQ)_0$ would then be the left-nullspace of a matrix whose rows correspond to a basis of $U_0^d$ and whose columns correspond to characters on $V^d$ whose simultaneous vanishing defines $W^d$.

Constructing such characters directly can be difficult, so we choose an alternative presentation. The columns of the matrix $\mathcal{M}$ will correspond to a set of characters on $V^d$ whose simultaneous vanishing defines the zero subspace. Then, in addition to rows corresponding to a basis for $U_0^d$, we also include rows corresponding to a basis of $W^d$.

\subsubsection{Characters on $V^d$}\label{subsubsec:chars}

We recall that $V^d$ is given by the direct sum $\oplus_{v \in T_d} \Q_v^\times/(\Q_v^\times)^2$. Our characters will respect this decomposition.

For a place $p \nmid 2\infty$, we will define a pair of characters $\chi_p,\overline{\ord_p}:\Q_p^\times/(\Q_p^\times)^2 \rightarrow \FF_2$ via $\overline{\ord_p}(\alpha) = \ord_p(\alpha) \pmod 2$ and $\chi_p(\alpha) = \left ( \frac{\alpha^\prime}{p} \right )$ where $\alpha^\prime = \frac{\alpha}{p^{\ord_p \alpha}}$.

For $v = \infty$, we define a single character $\overline{\ord_\infty}(\alpha):\RR^\times/(\RR^\times)^2\rightarrow \FF_2$ by  $$\overline{\ord_\infty}(\alpha) = \left \{ \begin{array}{cl} 0  & \text{if  } \alpha > 0  \\ 1   & \text{if } \alpha < 0 \end{array} \right . $$

For $v = 2$, we define a valuation character $\overline{\ord_2}$ as for other primes $p$ and a pair of characters $\chi_2,\chi_2^\prime:\Q_2^\times/(\Q_2^\times)^2 \rightarrow \FF_2$ via  $$\chi_2(\alpha) = \left \{ \begin{array}{cl} 0  & \text{if  } \alpha^\prime \equiv 5 \pmod{8}  \\ 1   & \text{if } \alpha^\prime \not \equiv 5 \pmod{8} \end{array} \right . \text{ and } \chi_2^\prime(\alpha) = \left \{ \begin{array}{cl} 0  & \text{if  } \alpha^\prime \equiv 3 \pmod{8}  \\ 1   & \text{if } \alpha^\prime \not \equiv 3 \pmod{8} \end{array} \right .,$$ where $\alpha^\prime = \frac{\alpha}{2^{\ord_2 \alpha}}$.

These characters all extend to characters of $V^d$ given by first projecting $V^d$ onto $ \Q_v^\times/(\Q_v^\times)^2$ and then applying the character on $\Q_v^\times/(\Q_v^\times)^2$. We will let $\Psi$ denote the union of these characters over all $v \in T_d$ and $\Psi_\text{Orig}$ denote the subset coming from $v \mid 2\Delta\infty$.

\subsection{The Rows of $\mathcal{M}$}\label{subsubsec:rowsMphi}

The rows of $\mathcal{M}$ will be of two types. The first type are dense rows corresponding to an $\FF_2$-basis $B_U$ for $U_0^d$. Since $\Delta$ is non-square and $T_d$ contains all places dividing $\Delta$, there is at least one $v \in T$ such that $\overline{\ord_v} \Delta \ne 0$. Fix one such $v$. If $v \nmid \infty$, then we may take $B_U  = \{-1\} \cup \{p | 2d\Delta_E : p \nmid v\}$. If $v \mid \infty$, then we take $B_U  = \{p | 2d\Delta_E \}$. This choice of $B_U$ fixes the subspace $U_0^d$ and therefore the choice of $\Sel_\phi(E^d/\Q)_0$ being presented. We denote the set of $p \in B_U$ such that $p \mid 2\Delta\infty$ by $B_{U,\text{Orig}}$.

The second type of rows will be those corresponding to the basis of $B_W$ of $W^d$. The subspace $W^d$ decomposes as a direct sum $W^d = \oplus_{v \in T_d} W_v^d$ and we choose $B_W$ to respect this decomposition; that is we start with a basis for each $W_v^d$ and lift these to a basis of $B_W$ of $W^d$ such that each $b \in B_W$ projects non-trivially into exactly one $W_v^d$. For convenience, we choose a basis for $W_v^d$ such that $\overline{\ord_v}$ is non-trivial on at most one basis element. If $p \mid d$ is a prime of type $3$, we explicitly choose our basis for $W_p^d = \Q_p^\times/(\Q_p^\times)^2$ to be $\{ \Delta, d(A+2\sqrt{B})\}$ as in Lemma \ref{lem:loccond}.

For each $b \in B_U \cup B_W$, the corresponding row in $\mathcal{M}$ is given by $(\psi(b))_{\psi \in \Psi}$. For notational convenience, we will sometimes refer to a row of $\mathcal{M}$ as an element of $B_U$ or $B_W$ and we will similarly refer to columns by the character in $\Psi$ associated to them.

\begin{proposition}
$\Sel_\phi(E^d/\QQ)_0$ is given by the left nullspace of $\mathcal{M}$.
\end{proposition}
\begin{proof}
It is easy to see that the sets $B_U$ and $B_W$ are each linearly independent. Suppose that some non-trivial linear combination of rows in $B_U \cup B_W$ sums to zero. We then have some $s \in \Q^\times/(\Q^\times)^2$ that can be expressed as both a non-trivial linear combination of rows in $B_U$ and as a non-trivial linear combination of rows in $B_W$. As a result, we find that $s \in   U_0^d$ and $s \in W^d$, and therefore that $s \in \Sel_\phi(E^d/\Q)_0$.

Now suppose that $s \in \Sel_\phi(E^d/\Q)_0$. Since $s \in U_0^d \cap W^d$, $s$ may therefore be expressed both as a unique linear combination of elements in $B_U$ and as a unique linear combination of elements of $B_W$. As a result, $s$ gives rise to a unique left-nullvector of $C$.
\end{proof}

The matrix $\mathcal{M}$ has a structural nullvector arising from the duality between $\Sel_\phi(E^d/\QQ)$ and $\Sel_{\hat\phi}(E^{\prime d}/\QQ)$. Let $(\cdot, \cdot)_v$ be the (additive) Hilbert symbol on $\Q_v^\times/(\Q_v^\times)^2$ and define a character $\chi:V^d \rightarrow \FF_2$ as $\chi(x) = \sum_{v \in T_d} (x,\Delta^\prime)_v$.

\begin{proposition}\label{prop:rightnullvector}
The character $\chi$ gives rise to a right nullvector of $\mathcal{M}$.
\end{proposition}
\begin{proof}
We consider the value of $\chi$ on the elements of $B_U$ and $B_W$ separately. We begin by noting that $\Delta^\prime \in \Sel_\phihat(E^{\prime d}/\Q) \subset \Q^\times/(\Q^\times)^2$ is the image of $(0,0) \in E^d(\QQ)$ under the map $\kappa^\prime_v$.

If $w \in B_W$, then $w$ projects non-trivially into exactly one $W_v^d$, so $\chi(w) = (w_v,\Delta^\prime)_v$, where $w_v$ is the image of $w$ in $W_v^d$. By design, $w_v$ is in the image of $\kappa_v$ and $\Delta^\prime$ is in the image of $\kappa^\prime_v$. Since the images of $\kappa_v$ and $\kappa_v^\prime$ are dual to each other via the Hilbert symbol pairing, we find that $\chi(w) = (w_v,\Delta^\prime)_v = 0$.

If $u \in B_U$, then the product formula for Hilbert symbols tells us that $\sum_{v \text{ of } \Q} (u, \Delta^\prime)_v = 0$. Since $\Delta^\prime$ and $u$ have trivial valuation for $v \not \in T_d$, we have $(u,\Delta^\prime)_v = 0$ for $v \not \in T_d$. As a result, we have $\chi(u) = \sum_{v \in T_d} (u, \Delta^\prime)_v = 0$.

Finally, we note that we are able to write $\chi$ as a sum of characters in $\Psi$ and as a result, $\chi$ gives rise to a dependency among the columns of $\mathcal{M}$.
\end{proof}

We conclude this section by analyzing the dimensions of $\mathcal{M}$.

\begin{lemma}\label{lem:rowcolcount}
The matrix $\mathcal{M}$ has $2|T_d| + \ord_2  \T(E^d/E^{\prime d})-1$ rows and $2|T_d|$ columns.
\end{lemma}
\begin{proof}
We first count columns. Each place $v \in T_d$ with $v \nmid 2\infty$ gives rise to two columns, $v = \infty$ gives a single column, and $v = 2$ yields three columns for a total of $2|T_d|$.

To count rows, we see that $|B_U| = |T_d|-1$ and that $|B_W| = \sum_{v \in T_d}  \dimF W_v^d$. 
We then have \begin{equation*}|B_U| + |B_W| = |T_d | -1 + \sum_{v \in T_d}  \dimF W_v^d = 2|T_d| -1 + \sum_{v \in T_d}  \left (\dimF W_v^d -1 \right ).\end{equation*}
The result then follows from Theorem \ref{prodform2}.
\end{proof}

\subsection{Surgery}\label{subsec:surg}
We would like to shrink $\mathcal{M}$ by removing some rows and columns while modifying others in a way that does not alter the dimension of the nullspace.
After surgery, we will be able to explicity describe most of the entries of the resulting matrix $\widehat{\mathcal{M}}$.
We will describe four different types of removals. The first removal method relies on the observation that if some column $c$ is dependent on the other columns of $\mathcal{M}$, then we may remove the column $c$ without affecting the dimension of the nullspace. The other three methods are premised on the observation that if a column $c$ has weight one, supported only on a row $r$, then any left nullvector of $\mathcal{M}$ must not contain $r$. We are therefore able to remove both the row $r$ and column $c$ from $\mathcal{M}$ without affecting the dimension of the nullspace.

\subsubsection{Dependent Column Removal}\label{subsubsec:depcolremove}

\begin{lemma}
There is a column $\overline{\ord_v}$ for some $v \mid 2\Delta\infty$ with $\overline{\ord_v}(\Delta) = 0$ such that $\overline{\ord_v}$ is dependent on the other elements on $\Psi$.
\end{lemma}

\begin{proof}
By Lemma 4.2 in \cite{K}, the assumption that $E$ does not have a cyclic 4-isogeny defined over $\QQ(E[2])$ ensures that neither $\Delta^\prime$ nor $\Delta\Delta^\prime$ is a square in $\Q^\times$. As a result, there exists some $v \mid 2\Delta\infty$ such that $\overline{\ord_v}(\Delta^\prime) = 1$ and $\overline{\ord_v}(\Delta) = 0$.

By Proposition \ref{prop:rightnullvector}, the character $\chi(x) = \sum_{v \in T_d} (x,\Delta^\prime)_v$ yields a dependency among the characters in $\Psi$. Since $\overline{\ord_v}(\Delta) = 1$, any attempt to write $(\cdot,\Delta^\prime)_v$ as a combination of elements of $\Psi$ must contain $\overline{\ord_v}$. As a result, the dependency includes $\overline{\ord_v}$.
\end{proof}

Since the column $\overline{\ord_v}$ is dependent on the other columns of $\mathcal{M}$, we may remove it without affecting the dimension of the nullspace.

\subsubsection{Special Column Removal}\label{subsubsec:specremove}

Our choice of $B_U$ ensures that there is some $v \in T$ such that $\overline{\ord_v}(B_U)=0$. In this instance, we have $\overline{\ord_v}(\Delta) = 1$ and the column $\overline{\ord_v}$ therefore has weight one supported on $w_v$ for some $w_v \in B_W$. We may therefore remove the column $\overline{\ord_v}$ and the row $w_v$ from $\mathcal{M}$ without affecting the dimension of the nullspace or any other entry in $\mathcal{M}$.

\subsubsection{Removal of Valuation Columns}\label{subsubsec:valcols}


We assume that have already performed the special removal described in Section \ref{subsubsec:specremove}. For any place $v$, we observe that the column $\overline{\ord_v}$ in $\mathcal{M}$ has weight one or two. In the event that $\overline{\ord_v}$ has weight one, the only row incident on $\overline{\ord_v}$ is the row $r_v \in B_U$ coming from $v \in T_d$. We are therfore able to remove both the column $\overline{\ord_v}$ and the row $r_v$ without affecting the dimension of the nullspace.


In the event that the column $\overline{\ord_v}$ has weight two, we observe that one of the rows is $r_v \in B_U$ and the other row is some $w_v \in B_W$ that restricts to a non-trivial basis element in $W_v^d$. If we add $w_v$ to $r_v$, then the column $\overline{\ord_v}$ will have weight one, supported only on $w_v$. We are then able remove the row $w_v$ and the column $\overline{\ord_v}$ as in the previous paragraph.

The result of this surgery is summarized by the following lemma.

\begin{lemma}\label{lem:surgval}\text{ }
\begin{enumerate}[(i)]
\item If $\overline{\ord_v}$ is trivial on $W_v^d$, then surgery to remove the column $\overline{\ord_v}$ removes the column $\overline{\ord_v}$ and the row $r_v \in B_U$, but does not otherwise alter $\mathcal{M}$.
\item If $\overline{\ord_v}$ is non-trivial on $W_v^d$, then surgery to remove the column $\overline{\ord_v}$ removes the column $\overline{\ord_v}$ and a row $w_v \in B_W$ with $\overline{\ord_v}(w_v) \ne 0$. It also replaces the row $r_v \in B_U$ with the row $r_v + w_v$. If $v \nmid \infty$, this may change the value of $r_v$ in the column $\chi_v$ (or $\chi_2^\prime$ if $p=2$) but does not otherwise alter $r_v$.
\end{enumerate}
\end{lemma}

\begin{remark}\label{rem:surgrem}
In the event that $\overline{\ord_p}$ is non-trivial on $W_p^d$ and $p \mid d$, then the value $\chi_p(p)$ in the column $\chi_p$ of the row $r_p$ is replaced by $\chi_p(p)$ + $\chi_p\left(d(A+2\sqrt{B}) \right)=  \legendrea{\frac{d}{p}(A + 2\sqrt{B})}{p}$.
\end{remark}

\subsubsection{Removal of Rows From $B_W$}\label{subsubsec:otherrows}

We are also able to remove any row $w \in B_W$ from $\mathcal{M}$. Suppose that $w_p$ is a row of $\mathcal{M}$ corresponding to an element of $B_W$ that projects non-trivially into $W_p^d$. 
For simplicity, assume that $\overline{\ord_p}(w_p) = 0$. If $p \ne 2$, this means that $w_p$ will be incident only on the column $\chi_p$. If $p = 2$,  $w_p$ may only be incident on $\chi_2$ and $\chi_2^\prime$. We treat these two cases separately.

Suppose that $p \ne 2$. While the column $\chi_p$ is trivial on all $w \in B_W$ that project trivially into $W_p^d$, unlike the column $\overline{\ord_p}$, $\chi_p$ may have entries on multiple rows in $B_U$. In the event that $\chi_p$ has weight one, then we may simply remove the row $w_p$ and the column $\chi_p$ without changing the dimension of the nullspace. Otherwise, we first need to add $w_p$ to each row $r \in B_U$ for which $\chi_p(r) \ne 0$ before we may remove the row $w_p$ and the column $\chi_p$. However, the row $w_p$ is supported entirely on the column $\chi_p$, so removing the row $w_p$ and column $\chi_p$ does not alter $\mathcal{M}$ beyond the removal of this row and column.

The story is similar if $p = 2$.  If $w_p$ is supported on at most one of $\chi_2$ and $\chi_2^\prime$, then we proceed exactly as in the case of $p \ne 2$ with the same results. Otherwise, we first need to choose which column to eliminate. For concreteness, we choose to eliminate $\chi_2^\prime$. To do so, we need to add $w_p$ to each row $r \in B_U$ for which $\chi_2^\prime(r) \ne 0$. Unlike the case where $p \ne 2$, the row $w_p$ has entries in two different columns. This results in altering the entry in column $\chi_2$ for each $r \in B_U$ for which $\chi_2^\prime(r) \ne 0$.

The result of this surgery is summarized by the following lemma.

\begin{lemma}\label{lem:surgchi}\text{ }
\begin{enumerate}[(i)]
\item If $p \ne 2$ and $w_p$ is supported entirely on $\chi_p$, then surgery to remove $w_p$ removes the row $w_p$ and the column $\chi_p$ but does not otherwise alter $\mathcal{M}$.
\item If $p = 2$ and the support of $w_p$ is contained in $\{\chi_2,\chi_2^\prime\}$, then surgery to remove $w_p$ removes the row $w_p$, one of $\chi_2$ and $\chi_2^\prime$, and does not otherwise alter the matrix outside of whichever of the columns $\chi_2$ and $\chi_2^\prime$ was not removed.
\end{enumerate}
\end{lemma}

\subsection{The Matrix $\widehat{\mathcal{M}}$}\label{subsec:mathatm}

We now use surgery to remove a subset of the rows  and columns of $\mathcal{M}$. We first perform the dependent column removal described in Section \ref{subsubsec:depcolremove} and then proceed with the special column removal described in Section \ref{subsubsec:specremove}. We next remove the column $\overline{\ord_v}$ for each $v \in T_d$ that was not affected by special removal or dependent column removal as described in Section \ref{subsubsec:valcols}. Finally, we remove all rows in $B_W$ as described in Section \ref{subsubsec:otherrows}. We call the resulting matrix $\widehat{\mathcal{M}}$.

\begin{proposition}\label{prop:dimsM}
The number of rows of $\widehat{\mathcal{M}}$ minus the number of columns of $\widehat{\mathcal{M}}$ is equal to $\ord_2 T(E^d/E^{\prime d})$.
\end{proposition}
\begin{proof}
With the exception of dependent column removal which only removes a column, the surgery process removes a row from $\mathcal{M}$ any time it removes a column and vice-versa. The result then follows from Lemma \ref{lem:rowcolcount}.
\end{proof}

Every row in the matrix $\widehat{\mathcal{M}}$ comes from some $p \in B_U$.  We denote the set of $p \in B_U$ which have an associated row in $\widehat{\mathcal{M}}$ by $\widehat{B_U}$ and let $\widehat{B_{U,\text{Orig}}} = \widehat{B_U} \cap B_{U,\text{Orig}}$. 
Every column in $\widehat{\mathcal{M}}$ comes from a character $\chi_p$ for $p \in T_d$ (to include $\chi_2^\prime$). We denote the set of characters $\chi_p$ which have an associated row in $\widehat{\mathcal{M}}$ by $\widehat{\Psi}$ and let $\widehat{\Psi_\text{Orig}}=\widehat{\Psi} \cap \Psi_\text{Orig}$. 
%
%
We now analyze when there is a row or column in $\widehat{\mathcal{M}}$ associated with a prime $p \mid d$ and what the values in that row or column are.
%
%
\begin{proposition}\label{prop:matstruct}
Suppose that $p \mid d$.
\begin{enumerate}[(i)]
\item If $\left (\frac{\Delta}{p} \right) = \left(\frac{\Delta^\prime}{p}\right ) = 1$ (i.e. $p$ is of type  $1$), then $p \in \widehat{B_U}$ and $\chi_p \in \widehat{\Psi}$.
\item If $\left (\frac{\Delta}{p} \right) = 1$ and $\left(\frac{\Delta^\prime}{p}\right ) = -1$  (i.e. $p$ is of type  $2$),  then $p \not \in \widehat{B_U}$ and $\chi_p \in \widehat{\Psi}$.
\item If $\left (\frac{\Delta}{p} \right) = -1$ and $\left(\frac{\Delta^\prime}{p}\right ) = 1$ (i.e. $p$ is of type  $3$), then $p \in \widehat{B_U}$ and $\chi_p \not \in \widehat{\Psi}$.
\item If $\left (\frac{\Delta}{p} \right) = \left(\frac{\Delta^\prime}{p}\right ) = -1$ (i.e. $p$ is of type  $4$), then $p \not \in \widehat{B_U}$ and $\chi_p \not \in \widehat{\Psi}$.
\end{enumerate}
%
%
In the event that $\chi_p \in  \widehat{\Psi}$, then for $q \in \widehat{B_U}$, the value of $\widehat{\mathcal{M}}_{q, \chi_p}$ is given by $$\widehat{\mathcal{M}}_{q, \chi_p} = \left \{ \begin{array}{cl}\legendrea{q}{p}  & \text{if  } p \ne q \\  \legendrea{\frac{d}{p}(A + 2\sqrt{B})}{p} & \text{if } p = q  \end{array} \right .$$

In the event that $p \in \widehat{B_U}$, then for $\psi \in \widehat{\Psi}$, the value of $\widehat{\mathcal{M}}_{p, \psi}$ is given by $$\widehat{\mathcal{M}}_{p, \psi} = \left \{ \begin{array}{cl} \psi(p) & \text{ if } \psi \in \widehat{\Psi_\text{Orig}} \\ \legendrea{p}{q}  & \text{if  } \psi = \chi_q \text{ for } q \mid \text{ with } q \ne p \\  \legendrea{\frac{d}{p}(A + 2\sqrt{B})}{p} & \text{if }\psi = \chi_p \end{array} \right .$$

%
%
\end{proposition}

\begin{proof}
By Lemma \ref{lem:surgchi}, the column $\chi_p$ will be removed from $\mathcal{M}$ during surgery if and only if there is some $w_p \in B_W$ supported entirely on $\chi_p$. By our choice of basis for $W_p^d$ in Section \ref{subsubsec:rowsMphi}, we see that this occurs precisely when $p$ is of type $3$ or $4$. We therefore have $\chi_p \in \widehat{\Psi}$ when $p$ is of types $1$ or $2$.

Suppose that $q \in \widehat{B_U}$. If the entry $\mathcal{M}_{q,\chi_p}$ was not altered via surgery, then we have $\widehat{\mathcal{M}}_{q,\chi_p} = \legendrea{q}{p}$. By Lemmas \ref{lem:surgchi} and \ref{lem:surgval}, we see that only instance in which the entry $\mathcal{M}_{q,\chi_p}$ is altered is when $q = p$, in which case $\widehat{\mathcal{M}_{p,\chi_p}} = \legendrea{\frac{d}{p}(A + 2\sqrt{B})}{p}$ as noted in Remark \ref{rem:surgrem}.

We next consider the rows. By Lemma \ref{lem:surgval}, a row $r_p$ coming from $p \in B_U$ will be removed from $\mathcal{M}$ during surgery exactly when $\overline{\ord_p}$ is trivial on a basis for $W_p^d$. By Corollary \ref{cor:locconds}, $r_p$ will not be removed -- and we therefore have $p \in \widehat{B_U}$ -- if $p$ is of type $1$ or type $3$.

Now suppose that $\psi \in \widehat{\Psi}$. If the entry $\mathcal{M}_{p,\psi}$ was not altered via surgery, then we have $\widehat{\mathcal{M}}_{p,\psi} = \psi(p)$, which if $\psi =\chi_q$ for $q \mid d$ with $p \ne q$ is equal to $\legendrea{q}{p}$. As above, the entry $\mathcal{M}_{p,\psi}$ is only altered when $q = p$, in which case $\widehat{\mathcal{M}_{p,\chi_p}} = \legendrea{\frac{d}{p}(A + 2\sqrt{B})}{p}$.
%
%
%
%
%
%
%
%
\end{proof}
We now sort the columns of $\widehat{\mathcal{M}}$ so that the columns in $\widehat{\Psi_\text{Orig}}$ are on the left and the columns in $\widehat{\Psi} \setminus \widehat{\Psi_\text{Orig}}$ are on the right. We further sort the columns of $\widehat{\Psi} \setminus \widehat{\Psi_\text{Orig}}$ to separate the primes of type $1$ from those of type $2$. We also sort the rows of $\widehat{\mathcal{M}}$ so that the rows in $\widehat{B_{U,\text{Orig}}}$ are on top and the rows in $\widehat{B_U} \setminus \widehat{B_{U,\text{Orig}}}$ are on the bottom. We further sort the rows in $\widehat{B_U} \setminus \widehat{B_{U,\text{Orig}}}$ to separate the primes of type $1$ from those of type $3$. The resulting matrix is shown in Figure \ref{fig:Mphi}


\begin{figure}
\captionsetup{labelformat=empty}

\begin{minipage}{\textwidth}

\vspace{0.4in}

\[ \vphantom{
    \begin{matrix}
    \overbrace{XYZ}^{\mbox{$R$}}%
     \\ \\ \\ \\ \\ \\
    \underbrace{pqr}_{\mbox{$S$}}
    \end{matrix}}%
\begin{matrix}
\vphantom{a}\\
\coolleftbrace{\widehat{B_{U,{\text{Orig }\hspace{0.1in}}}}}{\vspace{0.375in}}
\vspace{0.05in}\\
\coolleftbrace{\widehat{B_{U,{\text{Type 1}}}}}{\vspace{0.75in}}
\vspace{0.05in}\\
\coolleftbrace{\widehat{B_{U,{\text{Type 3}}}} }{\vspace{0.75in}}
\end{matrix}%
\renewcommand\arraystretch{1}
\begin{array}{c}
\begin{matrix}[ccc]
\coolover{\widehat{\Psi_\text{Orig}}}{\hspace{0.75in} \vspace{0.1in}} \hspace{0.1in}  %
&  \coolover{\widehat{\Psi_\text{Type 1}}}{\hspace{1.5in} \vspace{0.1in}}  \hspace{0.1in} & \coolover{\widehat{\Psi_\text{Type 2}}}{\hspace{1.5in} \vspace{0.1in}}\\
\end{matrix}

\vspace{-0.4in} \\

\left [
\begin{array}{p{0.75in}|p{1.5in}|p{1.5in}}
 \vspace{0.1in}{\hspace{0.15in}\Large $A_{1,1}$}  \vspace{0.2in} & \vspace{0.1in}{\hspace{0.45in}\Large $A_{1,2}$}  \vspace{0.2in} & \vspace{0.1in}{\hspace{0.45in}\Large $A_{1,3}$}  \vspace{0.2in} \\ \hline
\vspace{0.3in}{\hspace{0.15in}\Large $A_{2,1}$}  \vspace{0.35in} & \vspace{0.3in}{\hspace{0.45in}\Large $A_{2,2}$}  \vspace{0.35in} & \vspace{0.3in}{\hspace{0.45in}\Large $A_{2,3}$}  \vspace{0.35in} \\ \hline
\vspace{0.3in}{\hspace{0.15in}\Large $A_{3,1}$}  \vspace{0.35in} & \vspace{0.3in}{\hspace{0.45in}\Large $A_{3,2}$}  \vspace{0.35in} & \vspace{0.3in}{\hspace{0.45in}\Large $A_{3,3}$}  \vspace{0.35in} \\
\end{array}
\right ]
\end{array}
\]
\end{minipage}
\caption{Figure \ref{fig:Mphi}: The matrix $\widehat{\mathcal{M}}$}
\label{fig:Mphi}
\end{figure}


We now use Proposition \ref{prop:matstruct} to analyze the entries in each block of $\widehat{\mathcal{M}}$.

\begin{lemma}\label{lem:mostMhatentries} \text{ }
\begin{enumerate}[(i)]
\item The entries in $A_{1,2}$ (respectively $A_{1,3}$) are given by Legendre symbols $\legendrea{u}{p}$, where $u \in \widehat{B_{U,\text{Orig}}}$ and $p$ is a prime dividing $d$ of type $1$ (resp. type $2$).
\item The entries $A_{2,1}$  (respectively $A_{3,1}$) are given by characters $\psi(p)$, where $p$ is a prime of type $1$ (resp. type $3$) dividing $d$ and $\psi$ is either $\chi_2$, $\chi_2^\prime$, or a Legendre symbol $\legendrea{\cdot}{q}$, where $q$ is a prime dividing $\Delta$.
\item The entries in $A_{2,3}$ (respectively $A_{3,2}$, $A_{3,3}$) are of the form $\legendrea{q}{p}$ where $p$ is a prime of type $2$ (type $1$, type $2$) dividing $d$ and $q$ is a prime of type $1$ (resp. type $3$, type $3$).
\item The off-diagonal entries of the square matrix $A_{2,2}$ are of the form $\legendrea{q}{p}$ where $p \ne q$ are primes of type $1$ dividing $d$ and the diagonal entries are of the form $\legendrea{\frac{d}{p}(A + 2\sqrt{B})}{p}$ where $p$ is a prime of type $1$ dividing $d$.
\end{enumerate}
\end{lemma}
\begin{proof}
This is an immediate consequence of Proposition \ref{prop:matstruct}.
\end{proof}

\begin{lemma}\label{lem:A11}
The entries in $A_{1,1}$ are dependent solely on the class of $d$ in $\prod_{v \mid 2\Delta\infty} \Q_v^\times/(\Q_v^\times)^2$.
\end{lemma}
\begin{proof}
For a place $v \in T_d$, define $B_{W,v}$ to be the subset of $B_W$ that projects non-trivially into $W_v^d$ and set $B_{W,\text{Orig}} = \cup_{v \mid 2\Delta\infty} B_{W,v}$. Let $\mathcal{M}^\prime$ be the submatrix of $\mathcal{M}$ whose rows are given by the union $B_{U,\text{Orig}} \cup B_{W,\text{Orig}}$ and whose columns are given by the characters in $\Psi_\text{Orig}$. We note that the matrix $A_{1,1}$ is obtained by performing surgery (as described in Sections \ref{subsec:surg} and \ref{subsec:mathatm}) to the matrix $\mathcal{M}^\prime$ rather than the matrix $\mathcal{M}$.

We observe that the matrix $\mathcal{M}^\prime$ is dependent only on $E$ and the subspaces $W_v^d$ for $v \mid 2\Delta\infty$. By Proposition \ref{prop:tamcalc}, each $W_v^d$ is dependent only on the class of $d$ in $\Q_v^\times/(\Q_v^\times)^2$, and as a result, the dependence of $\mathcal{M}^\prime$ on $d$ is entirely dictated by the image of $d$ in $\prod_{v \mid 2\Delta\infty} \Q_v^\times/(\Q_v^\times)^2$. Since the outcome of surgery is determined solely by the input matrix, we therefore find that $A_{1,1}$ is dependent solely on the class of $d$ in $\prod_{v \mid 2\Delta\infty} \Q_v^\times/(\Q_v^\times)^2$.
\end{proof}

We then get the following result.
\begin{corollary}\label{cor:onlysymbols}
If $d = p_1p_2\cdots p_n$, then the rank of $\Sel_\phi(E^d/\QQ)$ depends only on
\begin{itemize}
\item The values of the $p_i \pmod {8\Delta}$,
\item The Legendre symbols $\legendre{p_i}{p_j}$ for those $p_i,p_j$ of types $1$, $2$, and $3$, and
\item The Legendre symbols $\legendre{A+2\sqrt{B}}{p_i}$ for those $p_i$ of type $1$.
\end{itemize}
\end{corollary}
\begin{proof}
This follows from Lemma \ref{lem:mostMhatentries} and Lemma \ref{lem:A11} combined with quadratic reciprocity.
\end{proof}

\section{A Probabilistic Approach}\label{sec:probapproach}

For this section, we will assume that $d$ is squarefree and that $(d,2\Delta) = 1$. For each $i$, we will let $n_i$ denote the number of prime factors of $d$ of type $i$.

Let $n = \omega(d)$ and suppose that $d = p_1p_2 \cdots p_n$. By Corollary \ref{cor:onlysymbols}, we know that $\dimF \Sel_\phi(E^d/\QQ)$ can be computed from an explicit matrix $\widehat{\mathcal{M}} = \widehat{\mathcal{M}}_d$ whose entries only depend on the values of the $p_i \pmod {8\Delta}$, the Legendre symbols $\legendre{p_i}{p_j}$ for those $p_i,p_j$ of types $1$, $2$, and $3$ , and the Legendre symbols $\legendre{A+2\sqrt{B}}{p_i}$ for those $p_i$ of type $1$.

There is a natural probability distribution $\mathcal{P}$ over possible combinations of such values. Namely, the $p_i$ take random, independent congruence classes in $(\Z/(8\Delta\Z))^\times$, the Legendre symbols $\legendre{p_i}{p_j}$ are random and independent up to the constraints imposed by quadratic reciprocity, and the Legendre symbols $\legendre{A+2\sqrt{B}}{p_i}$ for $p_i$ of type $1$ are randomly $+1$ or $-1$, each with probability $\frac{1}{2}$.

We may analyze the distribution of nullities of the matrices $\widehat{\mathcal{M}}_d$ -- and therefore of $\dimF \Sel_\phi(E^d/\QQ)$ -- as $d$ varies subject to the probability distribution $\mathcal{P}$. For the remainder of this section, we will therefore undertand the term probability to be speaking solely in terms of the probability distribution $\mathcal{P}$ and not with respect to any natural ordering on $d$.

%
We begin by defining a pair of probability distributions $\alpha_{r,u}(n)$ and $\alpha_{r,u}^\prime(n)$.
\begin{equation*}
\alpha_{r,u}(n):=\mathbb{P}\big(  \dimF\Sel_\phi(E^d)=r|d=p_1\cdots p_n, \dimF\Sel_\phi(E^d)-\dimF\Sel_\phi(E^{\prime d})=u \big)
\end{equation*}
and
\begin{equation*}
\alpha_{r,u}^\prime(n):=\mathbb{P}\left ( \begin{array}{c} \dimF\Sel_\phi(E^d)=r|d=p_1\cdots p_n, \dimF\Sel_\phi(E^d)-\dimF\Sel_\phi(E^{\prime d})=u, \\  \textrm{ and there are at least }n/10 \ p_i \textrm{ of type }i\textrm{ for each }i\end{array}  \right ).
\end{equation*}
%
%
The goal is then to prove the following theorem.
\begin{theorem}\label{thm:SDLimitThm}
For all $r,u$ we have
$$
\lim_{n\rightarrow\infty} \alpha_{r,u}(n) =  \lim_{n\rightarrow\infty} \alpha_{r,u}^\prime(n) = \alpha_{r,u}.
$$
\end{theorem}

Theorem \ref{thm:SDLimitThm} will be proved by studying the distribution of nullities of $ \widehat{\mathcal{M}}_d$. We begin by noting that $\lim_{n\rightarrow\infty} \alpha_{r,u}(n) =  \lim_{n\rightarrow\infty} \alpha_{r,u}^\prime(n)$.

\begin{proposition}\label{prop:lotseachtype}
The probability that $d$ has at least $\frac{n}{10}$ prime factors of type $i$ for each $i$ tends exponentially quickly to $1$ as $n \rightarrow \infty$.
\end{proposition}
\begin{proof}
By Lemma 4.2 in \cite{K}, the assumption that $E$ does not have a cyclic 4-isogeny defined over $\QQ(E[2])$ ensures that none of $\Delta$, $\Delta^\prime$, and $\Delta\Delta^\prime$ is a square in $\Q^\times$. As a result, the equidistribution of primes of each type follows from the assumption that the images of the primes $p_i$ are uniform and independent in $(\Z/(8\Delta\Z))^\times$. The proposition then follows from the Chernoff bounds.
\end{proof}

\begin{lemma}\label{lem:fullranksub}
Let $\widehat{\mathcal{M}}_d^\prime$ be the submatrix of $\widehat{\mathcal{M}}_d$ given by $\begin{bmatrix}A_{1,1} & A_{1,2} & A_{1,3} \\ A_{2,1} & A_{2,2} & A_{2,3}.\end{bmatrix}$ Then $\widehat{\mathcal{M}}_d^\prime$ has trivial left nullspace with probability $1-O(2^{-n_2})$.
\end{lemma}

\begin{proof}
Let $\widehat{\mathcal{M}}_d^{\prime\prime}$ be the submatrix of $\widehat{\mathcal{M}}_d^\prime$ given by $\begin{bmatrix}A_{1,2} & A_{1,3} \\ A_{2,2} & A_{2,3}.\end{bmatrix}$ and note that the nullspace of $\widehat{\mathcal{M}}_d^{\prime}$ is trivial if the nullspace of $\widehat{\mathcal{M}}_d^{\prime\prime}$ is trivial. We will show that the expected number of non-trivial left nullvectors of $\widehat{\mathcal{M}}_d^{\prime \prime}$ is $O_E(2^{-n_2})$.

Suppose that $w$ is a linear combination of rows in the submatrix $\begin{bmatrix} A_{1,2} & A_{1,3} \end{bmatrix}$ of $\widehat{\mathcal{M}}_d^{\prime\prime}$. The entries of $w$ are given by $\legendre{u}{p_i}$, where $u$ is some element in $\Z_T^\times$ and $p_i$ is a prime of type $1$ or $2$. Assuming the distribution $\mathcal{P}$, these entries are uniformly and independently distributed in $\{0,1\}$ and the probability that $w$ is trivial is therefore equal to $2^{-(n_1+n_2)}$.

Now suppose that $w$ is a linear combination of rows of $\widehat{\mathcal{M}}_d^{\prime\prime}$ containing at least one row $r$ of the submatrix $\begin{bmatrix} A_{2,2} & A_{2,3}\end{bmatrix}$. By design, $r = r_p$ for some prime $p$ of type $1$ dividing $d$ and $\widehat{\mathcal{M}}_d^{\prime\prime}$ therefore has a column $c_p$ corresponding to $\chi_p$. Let $\widehat{\mathcal{M}}_d^{\prime \prime \prime}$ be the submatrix of $\widehat{\mathcal{M}}_d^{\prime\prime}$ obtained by removing the column  $c_p$ and let $r^\prime$ and $w^\prime$ be the result of removing the column $c_p$ from $r$ and $w$. The entries in the row $r^\prime$ of $\widehat{\mathcal{M}}_d^{\prime \prime\prime}$ are independent of all the other entries in $\widehat{\mathcal{M}}_d^{\prime \prime\prime}$. Therefore the likelihood that $w^\prime$ is trivial is equal to $2^{-c}$, where $c = n_1 + n_2 -1$ is the number of columns of $\widehat{\mathcal{M}}_d^{\prime \prime\prime}$. As $w^\prime$ must be trivial for $w$ to be trivial, we find that $w$ is trivial with probability at most $2^{-(n_1+n_2)+1}$.

Finally, we observe that $\widehat{\mathcal{M}}_d^{\prime\prime}$ contains $n_1 + O_E(1)$ rows. There are therefore $O_E(2^{n_1})$ linear combinations of rows of $\widehat{\mathcal{M}}_d^{\prime\prime}$, each of which is trivial with probability at most $2^{-(n_1+n_2)+1}$. The expected number of nullvectors of $\widehat{\mathcal{M}}_d^{\prime\prime}$ is therefore bounded by $O_E(2^{-n_2})$.
\end{proof}

We are now ready to prove Theorem \ref{thm:SDLimitThm}.

\begin{proof}[Proof of Theorem \ref{thm:SDLimitThm}]
By Proposition \ref{prop:lotseachtype}, we are content to limit ourselves to proving $\lim_{n\rightarrow\infty} \alpha_{r,u}^\prime(n) = \alpha_{r,u}$.

Assuming distribution $\mathcal{P}$, the dimension of $\dimF \Sel_2(E^d/\QQ)_0$ is distributed like the nullity of $\widehat{\mathcal{M}}_d$, and as a result, the dimension of $\dimF \Sel_2(E^d/\QQ)$ is distributed like one plus the nullity of $\widehat{\mathcal{M}}_d$. We therefore wish to study this latter distribution.

Let $r$ be the number of rows and $c$ be the number of columns of of $\widehat{\mathcal{M}}_d$. By Proposition \ref{prop:dimsM}, we have $u = r - c = \ord_2 \T(E^d/E^{\prime d})$.

Let $\widehat{\mathcal{M}}_d^\prime$ and $A_3$ be the submatrices of  $\widehat{\mathcal{M}}_d$ given by $\widehat{\mathcal{M}}_d^\prime = \begin{bmatrix} A_{1,1} & A_{1,2} & A_{1,3} \\ A_{2,1} & A_{2,2} & A_{2,3}\end{bmatrix}$ and $A_3 = \begin{bmatrix} A_{3,1} & A_{3,2} & A_{3,3} \end{bmatrix}$. Assuming distribution $\mathcal{P}$, we find that the entries of $A_3$ are independently and uniformly distributed in $\FF_2$ and we observe that this will remain the case if we perform standard column operations on $A_3$.

Since $d$ has at least $\frac{n}{10}$ prime factors of type $n_2$, Lemma \ref{lem:fullranksub} tells us that $\widehat{\mathcal{M}}_d^\prime$ has full rank with exponentially high probability as $n \rightarrow \infty$. Conditioning on this event, we may therefore apply column operations to transform $\widehat{\mathcal{M}}_d^\prime$ into the matrix $\begin{bmatrix}I & 0 \end{bmatrix}$, where $I$ is an $m \times m$ identity matrix with $m = r - n_3$ rows and columns.

We may apply these same column operations to $\widehat{\mathcal{M}}_d$ to obtain $\begin{bmatrix} I \hspace{0.15in} 0 \\ A_3^\prime \end{bmatrix}$ and then use row operations to zero out the leftmost $m$ columns of $A_3^\prime$ yielding $\begin{bmatrix} I & 0 \\ 0 & A_3^{\prime\prime}\end{bmatrix}$, where $A_3^{\prime\prime}$ is an $n_3 \times c - m$ matrix with entries independently and uniformly distributed in $\FF_2$. We also observe that the dimension of the left nullspace of $\widehat{\mathcal{M}}_d$ is equal to that of $A_3^{\prime\prime}$.

As $m = r - n_3$, we find that $A_3^{\prime\prime}$ has $c - m = c - r + n_3 = n_3 - u$ columns. Theorem \ref{thm:randmatthm} then tells us that the probability that $A_3^{\prime \prime}$ -- and therefore $\widehat{\mathcal{M}}_d$ -- has nullity equal to $r-1$ tends to $\alpha_{r,u}$ exponentially quickly as $n_3 \rightarrow \infty$. Since we assumed that $d$ had at least $\frac{n}{10}$ prime factors of type $3$, we get the same result as $n \rightarrow \infty$.

As mentioned above, the dimension of $\dimF \Sel_2(E^d/\QQ)$ is one greater than the nullity of $\widehat{\mathcal{M}}_d$. As a result, the probability that $\dimF \Sel_2(E^d/\QQ)$ is equal to $r$ tends exponentially quickly to $\alpha_{r,u}$.
%
%
%
\end{proof}

\section{Natural Density}\label{sec:analytictech}

While Theorem \ref{thm:SDLimitThm} proves a limiting result along the lines of \cite{SD}, it would be convenient to have a result in terms of natural density such as Theorem \ref{MainThm} above. We proceed in a manner analogous to that in \cite{Kane} with a few added complications due to our slightly different context. In particular, we attempt to get at the densities of ranks via moments of the actual sizes of the Selmer groups in question on a large subset. In particular, letting $\omega(n)$ denote the number of distinct primes dividing $n$, we define:

\begin{defn}
Let $S'(X,u)$ be the set of $d$ satisfying the following properties:
\begin{enumerate}
\item $0<d\leq X$
\item $d$ is squarefree
\item $d$ is relatively prime to $2\Delta\Delta'$
\item $\dimF\Sel_\phi(E^d) - \dimF\Sel_\phi(E'^d) = u$
\item $|\omega(d)-\log\log(X)| < \log\log(X)^{5/8}$
\item $d$ has more than $\omega(d)/10$ prime factors of each type (ie. type 1 through type 4).
\end{enumerate}
Let $S''(X,u)$ be the set of $d$ satisfying only the first four of these properties.
\end{defn}

We note that $S'(X,u)$ is a proper subset of $S''(X,u)$, but that the density of one within the other approaches 1 as $X\rightarrow \infty$
\begin{lem}\label{densityLem}
For any $E$ and $u$ we have that
$$
\lim_{X\rightarrow\infty} \frac{|S'(X,u)|}{|S''(X,u)|} =1.
$$
\end{lem}
In order to prove this, we will need the following slight strengthening of \cite{Kane} Proposition 10:
\begin{prop}\label{averageDesnityProp}
Let $n,N,D$ be integers with $\log\log N>1$, and $(\log\log N)/2 < n < 2\log\log N$. Let $G=((\Z/D\Z)^*)^n$. Let $f:G\rightarrow\C$ be a function. Let $\bar{f} = \frac{1}{|G|}\sum_{g\in G}f(g)$. Let $|f|_2=\sqrt{\frac{1}{|G|}\sum_{g\in G}|f(g)|^2}.$ Then
$$
\frac{1}{n!}\sum_{\substack{p_1\cdots p_n \leq N \\ p_i \textrm{ distinct primes} \\ (p_i,D)=1}}f(p_1,\ldots,p_n) = \bar{f}\left(\frac{1}{n!}\sum_{\substack{p_1\cdots p_n \leq N \\ p_i \textrm{ distinct primes} \\ (p_i,D)=1}}1\right)+ O_D\left(\frac{|f|_2N\log\log\log N}{\log\log N} \right).
$$
\end{prop}
\begin{proof}
The result follows from the proof of \cite{Kane} Proposition 10.
\end{proof}
There is a particularly, nice version of this result when $f$ is symmetric.
\begin{cor}\label{averageDensityCor}
Let $n,N,D$ be integers with $\log\log N>1$, and $(\log\log N)/2 < n < 2\log\log N$. Let $G=((\Z/D\Z)^*)^n$. Let $f:G\rightarrow\C$ be a function symmetric in its inputs. For $d$ relatively prime to $D$ with $\omega(d)=n$, let $f(d)=f(p_1,\ldots,p_n)$, where $p_i$ are the prime factors of $d$. Let $\bar{f} = \frac{1}{|G|}\sum_{g\in G}f(g)$. Let $|f|_2=\sqrt{\frac{1}{|G|}\sum_{g\in G}|f(g)|^2}.$ Then
$$
\sum_{\substack{d\leq N\\ d\textrm{ squarefree}\\(d,D)=1}}f(p_1,\ldots,p_n) = \bar{f}\left(\sum_{\substack{d\leq N\\ d\textrm{ squarefree}\\(d,D)=1}}1\right)+ O_D\left(\frac{|f|_2N\log\log\log N}{\log\log N} \right).
$$
\end{cor}
\begin{proof}
This follows immediately from Proposition \ref{averageDesnityProp} upon noting that each such $d$ can be written as a product $p_1,\ldots,p_n$ in exactly $n!$ ways.
\end{proof}
We can now prove Lemma \ref{densityLem}.
\begin{proof}
We begin by showing that $S''(X,u)$ is reasonably big, in particular, that $|S''(X,u)|=\Omega(X\log\log(X)^{-1/2})$. Pick an $L$ modulo $D=4\Delta\Delta'$ so that it is possible to have $\dimF\Sel_\phi(E^d)-\dimF\Sel_\phi(E'^d)=u$ for some $d\equiv L\pmod{D}.$ By Proposition \ref{prop:tamcalc}, this will happen whenever $d\equiv L$ and $n_3-n_2$ is equal to some particular constant, $U$. In particular, this implies that $\legendre{\Delta\Delta'}{L}= (-1)^U.$

Consider the number $d\leq X$ with $d$ squarefree and $\omega(d)=n$ for some $|n-\log\log X|<\log\log(X)^{5/8}$ with $d\equiv L\pmod{D}$ and $n_3-n_2=U$. Note that whether or not this holds for such a $d$ depends only on the congruence classes of the primes dividing $d$ modulo $D$. Thus, if we define $f(p_1,\ldots,p_n)$ to be 1 if it holds and 0 otherwise, we may apply Corollary \ref{averageDensityCor}.

We note that if the $p_i$ are picked randomly modulo $D$ with probability $\Theta_{u}(\log\log(X)^{-1/2})$ that $n_3-n_2=U$ and that at least one prime is not of type 2 or 3. We note furthermore that upon fixing the values of all of the $p_i$ modulo $D$ except for one of type 1 or 4, there is a unique setting of the last prime modulo $D$ so that $d\equiv L\pmod{D}$. This setting is of type 1 or 4, since if $d'$ is the product of the other primes dividing $d$, then $\legendre{\Delta\Delta'}{d'}=(-1)^{n_2+n_3}=(-1)^U=\legendre{\Delta\Delta'}{L}$, and thus $\legendre{\Delta\Delta'}{L/d'}=1.$ Therefore $\bar{f}=\Theta_{D,u}(\log\log(X)^{-1/2})$, and thus $|f|_2=\Theta_{D,u}(\log\log(X)^{-1/4}).$ Hence, applying Corollary \ref{averageDensityCor} for each $n$ with $|n-\log\log(X)|<\log\log(X)^{5/8}$ we find that letting $S(X)$ be the set of $d$ satisfying Properties (1),(2) and (3) above that
\begin{align*}
|S''(X,u)| & = \Theta_{D,u}(\log\log(X)^{-1/2})|S(X)|+O_{D,u}\left(\frac{X\log\log\log{X}}{(\log\log X)^{5/8}} \right).
\end{align*}
We note that by a slight modification of \cite{Kane} Corollary 8, we can show that the number of $d\leq X$ with $|d-\log\log(X)|>\log\log(X)^{5/8} = \exp(-\Omega(\log\log(X)^{1/4})).$ Thus, $|S(X)|=\Omega_D(X)$, and thus
$$
|S''(X,u)| = \Omega_{D,u}(X\log\log(X)^{-1/2}).
$$

We have yet to show that $|S''(X,u)-S'(X,u)|$ is small. In particular, by the above, $o(X\log\log(X)^{-1/2})$ of integers less than $X$ fail to satisfy property (5). Of the numbers satisfying the Properties (1),(2),(3) and (5), and $\omega(d)=n$, we can apply Corollary \ref{averageDensityCor} to count the number that fail to satisfy Property (6) since it is clear that this property depends only on the prime factors modulo $D$. It is also clear that $\bar{f}=e^{-\Omega(n)}$ and that $|f|_2=\sqrt{\bar{f}}$. Therefore, we have that the number of $d$ failing Property (6) is $O_D(X e^{-\Omega(n)})$. Summing over all $n$ with $|n-\log\log(X)|<\log\log(X)^{5/8}$ tells us that the number of $d$ satisfying the first five properties but not the sixth is $O_D(X\log(X)^{-\Omega_D(1)}),$ which is also much smaller than $|S''(X,u)|$. This completes the proof.
\end{proof}

Having restricted ourselves, to $S'(X,u)$, we may now consider the average moments of twists of $E$ by elements of this set. In particular, the bulk of our work will be to prove the following proposition:
\begin{prop}\label{momentProp}
Let $k$ be a non-negative integer, and $u$ be an integer. Then
$$
\lim_{X\rightarrow\infty} \frac{\sum_{d\in S'(X,u)} |\Sel_\phi(E^d)|^k}{|S'(X,u)|} = \sum_{r=0}^\infty 2^{kr}\alpha_{r,u}.
$$
\end{prop}
Note that the limit above is exactly what you would expect if an $\alpha_{r,u}$-fraction of the $d\in S'(X,u)$ had $\dimF\Sel_\phi(E^d)=r$. Before proceeding with the proof, we show how Proposition \ref{momentProp} can be used to prove Theorem \ref{MainThm}.

\begin{proof}[Proof of Theorem \ref{MainThm} assuming Proposition \ref{momentProp}]
We proceed along the same lines as \cite{Kane}, Section 5 to show that for any $u,r$
$$
\lim_{X\rightarrow\infty} \frac{\#\{d\in S'(X,u): \dimF\Sel_\phi(E^d)=r\}}{|S'(X,u)|} = \alpha_{r,u}.
$$
Let
$$
\beta_{r,u}(X) = \frac{\#\{d\in S'(X,u): \dimF\Sel_\phi(E^d)=r\}}{|S'(X,u)|}.
$$
We have that
\begin{equation}\label{betaMomentEqn}
\lim_{X\rightarrow\infty} \sum_r 2^{rk}\beta_{r,u}(X) = \sum_r 2^{rk}\alpha_{r,u} = \sum_r 2^{rk}2^{-r^2+O_u(r)} = 2^{k^2/4+O_u(k)}
\end{equation}
for all integers $k\geq 0$.

Suppose that for some sequence of integers $X_1,X_2,X_3,\ldots$ that $\beta_{r,u}(X_i)\rightarrow \beta_{r,u}$ for all $r$ and some $\beta_{r,u}\in [0,1]$. We claim that $\beta_{r,u}=\alpha_{r,u}$ for all $r,u$ and note that this would complete our proof.

First note that for any $k\geq 0$ that by Equation \eqref{betaMomentEqn} that $\sum_r 2^{(k+1)r}\beta_{r,u}(X)$ is bounded by some $C_{r,u,k}$ depending on $r,u,k$, but not $X$. This allows us to apply the dominated convergence theorem to Equation \eqref{betaMomentEqn} along the sequence $X=X_i$ to conclude that for any $k\geq 0$ that
$$
\sum_r 2^{rk}\beta_{r,u} = \sum_r 2^{rk}\alpha_{r,u}=2^{k^2/4+O_u(k)}.
$$
Applying this with $k=r$, we find that
\begin{equation}\label{betaSizeEqn}\beta_{r,u} \leq 2^{-r^2+O_u(r)}.\end{equation}

Define the analytic functions $F(z)=\sum_r z^r\alpha_{r,u}$, $G(z)=\sum_r z^r\alpha_{r,u}.$ Note that by Equation \eqref{betaSizeEqn} that both are entire functions. Furthermore, it is easy to show that for $|z|\leq 2^a$ that $|F(z)|,|G(z)| \leq 2^{a^2/4+O_u(a)}.$ By Equation \eqref{betaMomentEqn}, $F(z)=G(z)$ when $z$ is a power of $2$. However Jensen's Theorem tells us that unless $F=G$ that the average value of $\log_2(|F(z)-G(z)|)$ over $|z|=2^a$ is
$$
O(1) + \sum_{\rho\textrm{ root of }F(z)-G(z), |\rho|<2^a}\log_2(2^a/|\rho|) \geq O(1) + \sum_{k=0}^{\lfloor a\rfloor} a-k \geq a^2/2 + O(1).
$$
This contradicts our bounds on $|F|$ and $|G|$. Therefore $F=G$ identically, and by comparing coefficients $\alpha_{r,u}=\beta_{r,u}$ for all $r$. Since this holds for taking $\beta_{r,u}$ to be the limits of $\beta_{r,u}(X_i)$ along any subsequence, it implies that
$$
\lim_{X\rightarrow\infty} \frac{\#\{d\in S'(X,u): \dimF\Sel_\phi(E^d)=r\}}{|S'(X,u)|} = \alpha_{r,u}
$$
for all $r,u$.

By Lemma \ref{densityLem}, it immediately follows that
$$
\lim_{X\rightarrow\infty} \frac{\#\{d\in S''(X,u): \dimF\Sel_\phi(E^d)=r\}}{|S''(X,u)|} = \alpha_{r,u}.
$$
Writing $S''(F,X,u)$ to denote the version of $S''$ associated to a perhaps different elliptic curve, $F$, we note that
$$
S(X,u) = \bigcup_{m|2\Delta\Delta'} m S''(E^m,X/m,u).
$$
Therefore, the set of twists of the form
$$
\{ E^d : d\in S(X,u) \}
$$
can be written as a union
$$
\bigcup_{m|2\Delta\Delta'} \{ (E^m)^d : d\in S(E^m,X,u) \}.
$$
Since an $\alpha_{r,u}$-fraction of the twists in each of these sets have Selmer groups of rank $r$, the same holds for the union. This completes the proof.
\end{proof}

The rest of this section will be devoted to proving Proposition \ref{momentProp}. We begin by further partitioning $S'(X,u)$ further. In particular let $S'(X,u,n)$ be the subset of $d\in S'(X,u)$ so that $\omega(d)=n$. We note that (as long as $|n-\log\log(X)|<\log\log(X)^{5/8}$) that each $d\in S'(X,u)$ can be written in exactly $n!$ ways as $d=p_1\cdots p_n$ where $p_i$ are distinct primes, relatively prime to $2\Delta\Delta'$, so that if $n_i$ of them are of type $i$, then $n_i>n/10$ and $n_3-n_2=U_d:=u-c_L$, where $c_L$ is as given in Proposition \ref{prop:tamcalc} for $d\equiv L\pmod{2\Delta\Delta'}$. Thus we have, letting $D=8\Delta\Delta'$, that
$$
\sum_{d\in S'(X,u,n)} |\Sel_\phi(E^d)|^k = \frac{1}{n!} \sum_{\substack{p_1,\ldots,p_n \textrm{ distinct primes}\\ d=p_1\cdots p_n\leq X, (p_i,D)=1 \\ n_i \textrm{ of type }i,  n_3-n_2=U_d\\ n_i > n/10 }} |\Sel_\phi(E^d)|^k.
$$
We subdivide this sum further by conditioning on the values of each of the $p_i$ modulo $D$. In particular, we let $C(u,n)$ be the set of elements $(c_1,\ldots,c_n)\in ((\Z/D\Z)^*)^n$ so that if there are $n_i$ $c$'s of type $i$, then $n_i>n/10$ for all $i$ and $n_3-n_2=U_d$. It is easy to verify that so long as $n>10U$ that $|C(u,n)| = \Theta(\phi(D)^n n^{-1/2} ).$ In any case, we can now rewrite the above equation as
\begin{equation}\label{congSplitEqn}
\sum_{d\in S'(X,u,n)} |\\Sel_\phi(E^d)|^k = \sum_{(c_1,\ldots,c_n)\in C(u,n)} \frac{1}{n!} \sum_{\substack{p_1,\ldots,p_n \textrm{ distinct primes}\\ d=p_1\cdots p_n\leq X \\ p_i \equiv c_i \pmod{D}}}|\Sel_\phi(E^d)|^k.
\end{equation}

We need to better understand $|\Sel_\phi(E^{p_1\cdots p_n})|$ when $p_i \equiv c_i \pmod{D}$. If $d=p_1\cdots p_n$, this is the number of $x\in U'^d$ so that $x\in W^d$. For each $i$, let $t(i,1),\ldots,t(i,n_i)$ be the distinct indices so $c_{t(i,j)}$ of type $i$. We note that any such $x$ can be written uniquely as
$$
x = y \prod_{i=1}^{n_1} p_{t(1,i)}^{u_i}\prod_{i=1}^{n_3} p_{t(3,i)}^{u_{n_1+i}},
$$
where $y$ is squarefree and divides $D$ and $u=(u_1,\ldots,u_{n_1+n_3})\in \F_2^{n_1+n_3}.$ We abbreviate the above as $x=yp^u.$

In order for $x$ to be in $W^d$ it must be the case that $x\in W^d_q$ for $q|D\infty$ and for $q=p_i$ for each $i$. We note that whether or not $x\in W^d_q$ for $q|D\infty$ depends only on the congruence classes of $p_i$ modulo $D$. Thus, if $c=(c_1,\ldots,c_n)$, we let $U(c)$ denote the set of pairs $(y,u)$ as above so that $yp^u$ is in $W^d_q$ for all $q|D\infty$. It should also be noted that such $x$ are automatically in $W^d_{p_i}$ for $i$ or type 3 or 4. For $p_i$ of type 1, $x\in W^d_{p_i}$ if and only if the Hilbert symbol $(x,d(A+2\sqrt{B}))_{p_i}$ equals 1. For $p_i$ of type 2, $x\in W^d_{p_i}$ if and only if the Hilbert symbol $(x,p_i)_{p_i}$ equals 1. Therefore, we have that if $x=yp^u$ for $(y,u)\in U(c)$ then if $p_i\equiv c_i \pmod{D}$ then
$$
\sum_{w\in \F_2^{n_1+n_2}} \prod_{i=1}^{n_1} (x,d(A+2\sqrt{B}))_{p_{t(1,i)}}^{w_i}\prod_{i=1}^{n_2} (x,p_{t(2,i)})_{p_{t(2,i)}}^{w_{n_1+i}} = \begin{cases} 2^{n_1+n_2} & \textrm{if }x\in W^d \\ 0 & \textrm{else}\end{cases}.
$$
Therefore we have that if $d=p_1\cdots p_n$ with $p_i \equiv c_i \pmod{D}$,
\begin{equation*}
|\Sel_\phi(E^d) | = \frac{1}{2^{n_1+n_2}}\sum_{\substack{(y,u)\in U(c) \\ w\in \F_2^{n_1+n_2}}}\prod_{i=1}^{n_1} (yp^u,d(A+2\sqrt{B}))_{p_{t(1,i)}}^{w_i}\prod_{i=1}^{n_2} (yp^u,p_{t(2,i)})_{p_{t(2,i)}}^{w_{n_1+i}}.
\end{equation*}
Taking a $k^{th}$ power yields
\begin{equation}\label{selmerSizeSumEqn}
|\Sel_\phi(E^d) |^k = \frac{1}{2^{k(n_1+n_2)}}\sum_{\substack{(y_\ell,u_\ell)\in U(c) \\ w_\ell\in \F_2^{n_1+n_2}\\ 1\leq \ell \leq k}} \prod_{\ell=1}^k\left(\prod_{i=1}^{n_1} (y_\ell p^{u_\ell},d(A+2\sqrt{B}))_{p_{t(1,i)}}^{w_{i,\ell}}\prod_{i=1}^{n_2} (y_\ell p^{u_\ell},p_{t(2,i)})_{p_{t(2,i)}}^{w_{n_1+i,\ell}}\right).
\end{equation}
Substituting this in to Equation \eqref{congSplitEqn}, and interchanging the order of summation, we get that
\begin{align}\label{legendreProdEqn}
\sum_{d\in S'(X,u,n)} & |\Sel_\phi(E^d)|^k = \sum_{(c_1,\ldots,c_n)\in C(u,n)} \frac{1}{2^{k(n_1+n_2)}}\sum_{\substack{(y_\ell,u_\ell)\in U(c) \\ w_\ell\in \F_2^{n_1+n_2}\\ 1\leq \ell \leq k}} \notag\\
& \frac{1}{n!} \sum_{\substack{ d=p_1\cdots p_n\leq X\\ p_i \textrm{ distinct primes} \\ p_i \equiv c_i \pmod{D}}} \prod_{\ell=1}^k\left(\prod_{i=1}^{n_1} (y_\ell p^{u_\ell},d(A+2\sqrt{B}))_{p_{t(1,i)}}^{w_{i,\ell}}\prod_{i=1}^{n_2} (y_\ell p^{u_\ell},p_{t(2,i)})_{p_{t(2,i)}}^{w_{n_1+i,\ell}}\right).
\end{align}
Define $\lambda(p)$ to be the function on primes $p$ relatively prime to $D$ so that
$$
\lambda(p) = \begin{cases} \legendre{A+2\sqrt{B}}{p} & \textrm{if } \legendre{\Delta}{p}=\legendre{\Delta'}{p}=1 \\ 0 & \textrm{otherwise.} \end{cases}
$$
We note that by quadratic reciprocity and our knowledge of the $p_i$ modulo $D$, we can rewrite the inner summand as
$$
z(y_\ell,u_\ell,w_\ell)\prod_{1\leq i < j \leq n} \legendre{p_i}{p_j}^{e_{i,j}(y_\ell,u_\ell,w_\ell)} \prod_{i\in T(y_\ell,u_\ell,w_\ell)} \lambda(p_i).
$$
For some $|z(y_\ell,u_\ell,w_\ell)|=1$, $e_{i,j}(y_\ell,u_\ell,w_\ell)=e_{j,i}(y_\ell,u_\ell,w_\ell)\in \F_2$, and $T(y_\ell,u_\ell,w_\ell) \subset\{1,\ldots,n\}$ so that $i\in T(y_\ell,u_\ell,w_\ell)$ only if $c_i$ is of type 1 and $e_{i,j}$ or $e_{j,i}$ is 1 for all $j$ of type 2.

We can now remove the conditioning on the congruence classes of the $p_i$ with an appropriate character sum. Namely, we have that:
\begin{align}\label{rightInnerSumEqn}
& \sum_{d\in S'(X,u,n)}  |\Sel_\phi(E^d)|^k = \frac{1}{\phi(D)^n}\sum_{(c_1,\ldots,c_n)\in C(u,n)}\frac{1}{2^{k(n_1+n_2)}} \sum_{\substack{(y_\ell,u_\ell)\in U(c) \\ w_\ell\in \F_2^{n_1+n_2}}} \sum_{\chi_i \pmod{D}}\notag\\
& \frac{1}{n!} \sum_{\substack{ d=p_1\cdots p_n\leq X \\ p_i \textrm{ distinct primes} }}z(y_\ell,u_\ell,w_\ell)\prod_{i=1}^n\chi_i(p_i/c_i)\prod_{1\leq i < j \leq n} \legendre{p_i}{p_j}^{e_{i,j}(y_\ell,u_\ell,w_\ell)} \prod_{i\in T(y_\ell,u_\ell,w_\ell)} \lambda(p_i).
\end{align}

The inner summand is now a constant of norm 1 times a product of $\chi_i(p_i)$ where the $\chi_i$ are characters of modulus dividing $D$, times a product of Legendre symbols $\legendre{p_i}{p_j}$ for $i<j$, times a product of terms of the form $\lambda(p_i)\legendre{p_i}{p_j}$ where $j=t(2,1)$. Note that this sum is very similar to the sum consider in \cite{Kane} Proposition 9, and can be bounded by similar means. In particular, we have
\begin{lem}\label{charsumLem}
Let $\chi_i$, $z=z(y_\ell,u_\ell,w_\ell)$, $e_{i,j}=e_{i,j}(y_\ell,u_\ell,w_\ell)$ and $T=T(y_\ell,u_\ell,w_\ell)$ be as above. Let $m$ be the number of indices $1\leq i \leq n$ so that at least one of the following holds:
\begin{itemize}
\item $\chi_i\neq 1$
\item $e_{i,j}=1$ for some $j\neq i$
\end{itemize}
Then
$$
\left| \frac{1}{n!} \sum_{\substack{ d=p_1\cdots p_n\leq X \\ p_i \textrm{ distinct primes} }}z\prod_{i=1}^n\chi_i(p_i/c_i)\prod_{1\leq i < j \leq n} \legendre{p_i}{p_j}^{e_{i,j}} \prod_{i\in T} \lambda(p_i) \right| = O_{c,D}(Xc^m).
$$
\end{lem}
\begin{proof}
We may assume without loss of generality that $c_n$ is of type 2. Thus, we may merge terms to replace the $\lambda(p_i)$ terms with terms of the form $\lambda(p_i)\legendre{p_i}{p_n}$. The remainder of the proof is now completely analogous to the proof of Proposition 9 in \cite{Kane} after noting that \cite{Kane} Lemma 15 also implies that
$$
\left|\sum_{\substack{A\leq p_1,p_2\\ p_1p_2\leq X}}a(p_1)b(p_2)\lambda(p_1)\legendre{p_1}{p_2} \right| = O(X\log(X) A^{-1/8}).
$$
\end{proof}

For given values of $\chi_i$,$y_\ell,u_\ell,w_\ell$, let $m$ be as given in Lemma \ref{charsumLem}, and let $m'$ be the number of indices $i$ so that $e_{i,j}=1$ for some $j\neq i$. We would like to show the contribution to the sum in Equation \eqref{rightInnerSumEqn} with $m>0$ is negligible. We begin by showing that the sum over terms with $m'>0$ is negligible. In particular, we show that
\begin{lem}\label{mPrimeZeroLem}
\begin{align*}
& \frac{1}{\phi(D)^n}\sum_{(c_1,\ldots,c_n)\in C(u,n)}\frac{1}{2^{k(n_1+n_2)}} \sum_{\substack{(y_\ell,u_\ell)\in U(c) \\ w_\ell\in \F_2^{n_1+n_2} \\ m'>0}} \sum_{\chi_i \pmod{D}}\notag\\
& \left|\frac{1}{n!} \sum_{\substack{ d=p_1\cdots p_n\leq X \\ p_i \textrm{ distinct primes} }}z(y_\ell,u_\ell,w_\ell)\prod_{i=1}^n\chi_i(p_i/c_i)\prod_{1\leq i < j \leq n} \legendre{p_i}{p_j}^{e_{i,j}(y_\ell,u_\ell,w_\ell)} \prod_{i\in T(y_\ell,u_\ell,w_\ell)} \lambda(p_i)\right| \\
& \ \ \ \ \ \ \ \ \ \ = O_{D,k,U}\left( X \log(X)^{-2^{-k-1}}\right).
\end{align*}
Furthermore, for fixed $c$, the number of collections of $u_\ell,v_\ell$ so that $m'=0$ is $O_{k,u}(2^{k(n_1+n_2)})$.
\end{lem}
\begin{proof}
To understand the size of this sum, we must better understand the number of $u_\ell,w_\ell$ with a given value of $m'$. In order to do this, we must better understand the terms $e_{i,j}$. We begin with the following definitions:

\begin{itemize}
\item For $1\leq i\leq n_1$, let $v_{1,i} \in \F_2^{2k}$ be given by $(w_{i,1},\ldots,w_{i,\ell},u_{i,1},\ldots,u_{i\ell})$.
\item For $1\leq i \leq n_2$, let $v_{2,i}\in \F_2^k$ be given by $(w_{n_1+i,1},\ldots,w_{n_1+i,\ell})$.
\item For $1\leq i \leq n_3$, let $v_{3,i}\in \F_2^k$ be given by $(u_{n_1+i,1},\ldots,u_{n_1+i,\ell})$.
\end{itemize}
It is now easy to verify that:
$$
e_{t(2,i),t(3,j)} = \langle v_{2,i}, v_{3,j} \rangle,
$$
and
$$
e_{t(1,i),t(1,j)} = \phi(t(1,i)+t(1,j))
$$
where $\phi$ is the non-degenerate quadratic form $\phi(x_1,\ldots,x_k,y_1,\ldots,y_k)=\sum_{i=1}^k x_iy_i.$

Call an index, $i$ between $1$ and $n$ \emph{active} if $e_{i,j}=1$ for any $j\neq i$. Let $S_1\subset \F_2^{2k}$ be the set of elements of the form $v_{1,i}$ for $i$ so that $t(1,i)$ is not active. Let $m_i$ be the number of active indices of type $i$. Define $S_2,S_3\subset \F_2^k$ similarly. We make the following claim:
\begin{claim}
$$
|S_1| \leq 2^k, \ \ \ \ \ |S_2||S_3| \leq 2^k
$$
Furthermore, the first inequality is strict if $m_1>0$ and the second inequality is strict if $m_2$ or $m_3$ is bigger than 0. Finally $m_4>0$ only if $m_1>0$.
\end{claim}
\begin{proof}
The first inequality follows from noting that for any $v_1,v_2\in S_1$ that $\phi(v_1+v_2)=0$, and thus that $S_1$ is contained in a translation of a Lagrangian subspace of $\phi$. If $m_1>0$ then there is some $t(1,i)$ which is active, and thus $v_{1,i}\not\in S_1$. On the other hand, by the above reasoning $S_1\cup\{v_{1,i}\}$ is contained in a translate of a Lagrangian subspace for $\phi$, implying that the inequality is strict.

The second inequality follows from the observation that $S_2$ is contained in the orthogonal compliment of the span of $S_3$. If $e_{t(2,i),t(3,j)}=1$ for some $i,j$, then $S_2$ is also orthogonal to $v_{3,j}\not\in S_3,$ from which we infer that either $S_3$ is strictly contained in it's span, or that $S_2$ is strictly contained in the orthogonal complement of $S_3$, either of which imply that $|S_2||S_3| < 2^k.$

Finally, note that $e_{t(4,i),t(4,j)}$ is always 0, and thus if $m_4>0$ then some other $m_i$ must also be positive.
\end{proof}
Note that this claim immediately implies the second part of the Lemma.

We are now ready to prove our Proposition. We write the sum over $u_\ell,w_\ell$ in a particular way. First we produce an outer sum over the values of $m_1,m_2,m_3,m$. Next we sum over possible choices of the sets $S_1,S_2,S_3$ consistent with the above claim. We note that there are only $O_k(1)$ many possibilities. Then we count the number of choices of $u_\ell,w_\ell,y_\ell$ consistent with these choices. We note that for each choice of $u_\ell,w_\ell$ there are $O_{k,D}(1)$ possible valid choices for $y_\ell$. We note that making choices of $u_\ell$ and $w_\ell$ is equivalent to picking values for the $v_{i,j}$. To do this we first decide which of the indices contribute to $m$, which can be done in at most $\binom{n}{m}$ many ways. Next, we pick the values of the $v_{i,j}$ consistently with our choices of $S_i$, which can be done in at most $|S_1|^{n_1}|S_2|^{n_2}|S_3|^{n_3}2^{km'}$ many ways. Finally, we note that  By Lemma \ref{charsumLem}, the inner sum is then $O_{k,c,D}(Xc^{m'})$. Finally, we choose the values of $\chi_i$, noting that $\chi_i=1$ unless $i$ contributes to $m$. Thus, the $\chi$
s can be picked in at most $\phi(D)^m$ ways. Thus the sum in question is at most
\begin{align*}
& \frac{1}{\phi(D)^n} \sum_{c\in C(n,u)} \frac{1}{2^{k(n_1+n_2)}} \sum_{0<m_1+m_2+m_3\leq m} O_{k,D,c}(X(2^k \phi(D) c)^{m})\binom{n}{m}|S_1|^{n_1}|S_2|^{n_2}|S_3|^{n_3}\\
\leq & \sum_{0<m_1+m_2+m_3\leq m} O_{k,D,c,U}(X(2^k \phi(D) c)^{m})\binom{n}{m}\left(\frac{|S_1|}{2^k} \right)^{n_1}\left(\frac{|S_2||S_3|}{2^k}\right)^{n_2}\\
\leq & (1-2^{-k})^{\min(n_1,n_2)}\sum_{m} O_{k,D,c,U}(X( c)^{m})\binom{n}{m}m^3\\
\leq & (1-2^{-k})^{n/10} O_{c,k,D,U}(X(1+c)^{3n})\\
\leq & (1-2^{-k})^{n/10} O_{k,D,U}(X(1+2^{-k-7})^n)\\
\leq & O_{k,d,U}(X\log(X)^{-2^{-k-5}}).
\end{align*}

This completes our proof.
\end{proof}

Now that we have shown that the contribution from terms with $m'>0$, we can deal with the sum in question.
\begin{lem}\label{fixedNLem}
For $|n-\log\log(X)|<\log\log(X)^{5/8}$,
$$
\sum_{d\in S'(X,u,n)}|\Sel_\phi(E^d)|^k = |S'(X,u,n)|\sum_r 2^{kr}\alpha_{r,u}(n) + O_{D,k,U}\left( \frac{X\log\log\log(X)}{\log\log(X)^{5/4}}\right).
$$
Furthermore,
$$
\sum_r 2^{kr}\alpha_{r,u}(n) = O_{D,k,U}(1).
$$
\end{lem}
\begin{proof}
By Lemma \ref{mPrimeZeroLem}, we know that we can already safely ignore the terms with $m'>0$. Also by Lemma \ref{mPrimeZeroLem}, the number of such terms in the sum over $y_\ell,u_\ell,v_\ell$ is $O_{k,D}(2^{k(n_1+n_2)})$. Thus, up to negligible error the sum in question without the $m>0$ restriction is
\begin{equation}\label{fEquation}
\frac{1}{n!} \sum_{\substack{ d=p_1\cdots p_n\leq X \\ p_i \textrm{ distinct primes} }}f(p_1,\ldots,p_n),
\end{equation}
where $f:((\Z/D\Z)^*)^n\rightarrow \C$ is some function with $|f|_\infty\leq O_{k,D}(1)$ and $f$ supported on $C(u,n)$.
By Proposition \ref{averageDesnityProp}, this is
$$
\left(\frac{1}{\phi(D)^n} \sum_{g\in((\Z/D\Z)^*)^n} f(g) \right)\left(\frac{1}{n!} \sum_{\substack{ d=p_1\cdots p_n\leq X \\ p_i \textrm{ distinct primes} }}1 \right)+O_D\left(\frac{X\log\log\log(X)|f|_\infty}{\log\log(X)}\sqrt{\frac{|\textrm{supp}(f)|}{\phi(D)^n}} \right).
$$
The error term here is clearly seen to be
$$
O_{D,k,U}\left( \frac{X\log\log\log(X)}{\log\log(X)^{5/4}}\right).
$$

First we note that $\sum_r 2^{kr}\alpha_{r,u}(n)$ is the expectation over $p_i$ as described in Theorem \ref{thm:SDLimitThm} of the $k^{th}$ moment of the size of the Selmer group times the indicator function of the event that there are more than $n/10$ primes of each type, given that $n_3-n_2=U$. This is the expectation of the $k^{th}$ power of Selmer times the indicator function that all $n_i$ are more than $n/10$ and that $n_3-n_2=U$, divided by the probability that $n_3-n_2=U$.

The former expectation can be computed via a formula similar to Equation \eqref{legendreProdEqn} in which the inner sum and the $\frac{1}{n!}$ is replaced by an expectation. In this case, the sum over terms with $m'>0$ is exactly 0, and thus is equal to an expression analogous to that in Equation \eqref{fEquation}. Thus, it is easy to see that this sum is exactly
$$
\left(\frac{1}{\phi(D)^n} \sum_{g\in((\Z/D\Z)^*)^n} f(g) \right).
$$
Thus, we have that
$$
\sum_r 2^{kr}\alpha_{r,u}(n) = \left(\frac{1}{C'(n,u)} \sum_{g\in((\Z/D\Z)^*)^n} f(g) \right)
$$
Where $C'(n,u)$ is the set of congruence classes with $n_3-n_2=U$. This is clearly $O_{D,k,U}(1)$.

By Proposition \ref{averageDesnityProp} with $f$ the indicator function of $C(u,n)$, we find that that $|S'(X,u,n)|$ is
$$
\#\{d\leq X \textrm{ squarefree },(d,D)=1,\omega(d)=n \}\frac{|C(u,n)|}{\phi(D)^n} + O_{D,k,U}\left( \frac{X\log\log\log(X)}{\log\log(X)^{5/4}}\right).
$$
Combining these last two lines with Equation \eqref{fEquation} yields the desired result.
\end{proof}

We are almost ready to prove Proposition \ref{momentProp}. First we need one more Lemma.

\begin{lem}\label{convergeLem}
For any $k,u$
$$
\lim_{n\rightarrow\infty} \sum_r 2^{kr} \alpha_{r,u}(n) = \sum_r 2^{kr} \alpha_{r,u}.
$$
\end{lem}
\begin{proof}
Applying the second part of Lemma \ref{fixedNLem}, with one higher $k$, we find that
$$
\sum_r 2^{kr+r} \alpha_{r,u}(n) = O_{k,D,U}(1)
$$
and thus
$$
2^{kr}\alpha_{r,u}(n) = O_{D,k,U}(2^{-r}).
$$
Theorem \ref{thm:SDLimitThm} tells us that $2^{kr}\alpha_{r,u}(n)$ converge to $2^{kr}\alpha_{r,u}$ pointwise. Our result now follows from the Dominated Convergence Theorem.
\end{proof}

We are now ready to prove Proposition \ref{momentProp}.
\begin{proof}
By Lemma \ref{convergeLem}, for any $\epsilon>0$ there is an $N$ so that whenever $n>N$, $|\sum_r 2^{kr} \alpha_{r,u}(n)- \sum_r 2^{kr} \alpha_{r,u}| < \epsilon$. Take $X$ so that $\log\log(X)>2N$. Then
\begin{align*}
& \sum_{d\in S'(X,u)} |\Sel_\phi(E^d)|^k \\
= & \sum_{|n-\log\log(X)|<\log\log(X)^{5/8}} \sum_{d\in S'(X,u,n)}|\Sel_\phi(E^d)|^k\\
= & \sum_{|n-\log\log(X)|<\log\log(X)^{5/8}} \left( |S'(X,u,n)|\sum_r 2^{kr}\alpha_{r,u}(n) + O_{D,k,U}\left( \frac{X\log\log\log(X)}{\log\log(X)^{5/4}}\right) \right)\\
= & \sum_{|n-\log\log(X)|<\log\log(X)^{5/8}} \left( |S'(X,u,n)|\sum_r 2^{kr}\alpha_{r,u} + O(\epsilon) \right) + O_{D,k,U}\left( \frac{X\log\log\log(X)}{\log\log(X)^{5/8}}\right)\\
= & |S'(X,u)|\sum_r 2^{kr}\alpha_{r,u} + O(\epsilon |S'(X,u)|) +O_{D,k,U}(|S'(X,u)|\log\log(X)^{-1/10}).
\end{align*}
Thus for $X$ sufficiently large,
$$
\frac{\sum_{d\in S'(X,u)} |\Sel_\phi(E^d)|^k}{|S'(X,u)|} = O(\epsilon).
$$
This completes the proof.
\end{proof}

\bibliography{citations}

\end{document}